\documentclass[11pt]{amsart} \usepackage{latexsym, amssymb, stmaryrd}
\usepackage[T1]{fontenc}

\DeclareTextSymbol{\thh}{T1}{254}
\def\th{\textnormal{\thh}}

\newtheorem{thm}{Theorem}[section]
\newtheorem{lemma}[thm]{Lemma}
\newtheorem{prop}[thm]{Proposition}
\newtheorem{cor}[thm]{Corollary}

\theoremstyle{definition}
\newtheorem{df}[thm]{Definition}
\newtheorem{rmk}[thm]{Remark}
\newtheorem{rmks}[thm]{Remarks}

\newtheorem{fact}[thm]{Fact}
\newtheorem{facts}[thm]{Facts}
\newtheorem{question}[thm]{Question}

\ifx\pdfoutput\@undefined\usepackage[usenames,dvips]{color}
\else\usepackage[usenames,dvipsnames]{color}
\IfFileExists{pdfcolmk.sty}{\usepackage{pdfcolmk}}{} 
\fi
\definecolor{isaac}{rgb}{0,0.5,0}
\definecolor{error}{rgb}{0.8,0,0}
\definecolor{clifton}{rgb}{0,0,0.8}

\newcommand{\Z}{\mathbb{Z}}

\newcommand{\curly}[1]{\mathcal{#1}}

\newcommand{\B}{\curly{B}}

\newcommand{\K}{\curly{K}}
\newcommand{\m}{\mathcal{M}}

\newcommand{\bo}[1]{\boldsymbol{#1}}

\makeatletter

\def\indsym#1#2{%
  \setbox0=\hbox{$\m@th#1x$}%
  \kern\wd0%
  \hbox to 0pt{\hss$\m@th#1\mid$\hbox to 0pt{$\m@th#1^{#2}$}\hss}%
  \lower.9\ht0\hbox to 0pt{\hss$\m@th#1\smile$\hss}%
  \kern\wd0}
\newcommand{\ind}[1][]{\mathop{\mathpalette\indsym{#1}}}

\def\nindsym#1#2{%
  \setbox0=\hbox{$\m@th#1x$}%
  \kern\wd0%
  \hbox to 0pt{\hss$\m@th#1\not$\kern1.4\wd0\hss}
  \hbox to 0pt{\hss$\m@th#1\mid$\hbox to 0pt{$\m@th#1^{\,#2}$}\hss}%
  \lower.9\ht0\hbox to 0pt{\hss$\m@th#1\smile$\hss}%
  \kern\wd0}
\newcommand{\nind}[1][]{\mathop{\mathpalette\nindsym{#1}}}

\def\dotminussym#1#2{%
  \setbox0=\hbox{$\m@th#1-$}%
  \kern.5\wd0%
  \hbox to 0pt{\hss\hbox{$\m@th#1-$}\hss}%
  \raise.6\ht0\hbox to 0pt{\hss$\m@th#1.$\hss}%
  \kern.5\wd0}
\newcommand{\dotminus}{\mathbin{\mathpalette\dotminussym{}}}

\def \K {{\mathcal K}}
\def \r { {\mathbb R} }
\def \<{\langle}
\def \>{\rangle}

\def \z {{\mathbb Z}}
\def \*Z {{{^*}\Z}}

\def \((  {(\!(}
\def \)) {)\!)}

\def \tp{\operatorname{tp}}

\def \acl{\operatorname{acl}}
\def \dcl{\operatorname{dcl}}

\def \p{\mathbb{P}}

\def \int{\operatorname{int}}
\numberwithin{equation}{section}
\def \l{\llbracket}
\def \rr{\rrbracket}
\def\thind{\ind[\th]}
\def \ethind{\ind[\th,\epsilon]}
\def\nthind{\nind[\th]}

\def \z{\mathcal{Z}}
\def \i{\operatorname{Ind}}
\def \acl{\operatorname{acl}}
\def \bdd{\operatorname{bdd}}
\def \Aut{\operatorname{Aut}}

\def \eq{\operatorname{eq}}
\def \feq{\operatorname{feq}}
\def \k{\mathcal{K}}
\def \real{\operatorname{real}}
\def \U{\mathfrak{U}}
\def \UU{\mathbb{U}}

\allowdisplaybreaks[2]

\begin{document}

\title{Thorn-forking in Continuous Logic}

\author{Clifton Ealy and Isaac Goldbring}

\address {Western Illinois University, Department of Mathematics,
476 Morgan Hall, 1 University Circle,
Macomb, IL 61455}
\email{CF-Ealy@wiu.edu}

\address {University of California, Los Angeles, Department of Mathematics, 520 Portola Plaza, Box 951555, Los Angeles, CA 90095-1555, USA}
\email{isaac@math.ucla.edu}
\urladdr{www.math.ucla.edu/~isaac}

\begin{abstract}
We study thorn forking and rosiness in the context of continuous logic.  We prove that the Urysohn sphere is rosy (with respect to finitary imaginaries), providing the first example of an essentially continuous unstable theory with a nice notion of independence.  In the process, we show that a real rosy theory which has weak elimination of finitary imaginaries is rosy with respect to finitary imaginaries, a fact which is new even for classical real rosy theories.  
\end{abstract}
\maketitle

\section{Introduction}

In classical model theory, thorn forking independence was defined by Tom Scanlon, and investigated by Alf Onshuus and then by the first author as a common generalization of forking independence in stable theories and (all known) simple theories as well as the independence relation in o-minimal theories given by topological dimension.  More generally, a theory $T$ is called rosy if thorn independence is a strict independence relation for $T^{\eq}$.  If $T$ is rosy, then thorn independence is the weakest notion of independence for $T^{\eq}$.  It thus follows that all simple theories and o-minimal theories are rosy.  It is the purpose of this paper to define and investigate thorn independence and rosiness in the context of \emph{continuous logic}.

Continuous logic is a generalization of first-order logic which is suited for studying structures based on complete metric spaces, called \emph{metric structures}.  Moreover,  one has continuous versions of nearly all of the notions and theorems from classical model theory.  In particular, stable theories have been studied in the context of continuous logic; see \cite{BU}.  Nearly all of the ``essentially continuous'' theories that were first studied in continuous logic are stable, e.g. infinite-dimensional Hilbert spaces, atomless probability algebras, $L^p$-Banach lattices, and richly branching $\r$-trees; see \cite{BBHU} and \cite{Carlisle}.  Here, ``essentially continuous'' is a vague term used to eliminate classical first-order structures, viewed as continuous structures by equipping them with the discrete metric, from the discussion.  One can also make sense of the notion of a continuous simple theory; see \cite{B4}, where the notion of simplicity is studied in the more general context of compact abstract theories.  However, there are currently no ``natural'' examples of an essentially continuous, simple, unstable theory and all attempts to produce an essentially continuous, simple, unstable theory have failed.  For example, adding a generic automorphism to almost all known essentially continuous stable theories (e.g. infinite-dimensional Hilbert spaces, structures expanding Banach spaces, probability algebras) yields a theory which is once again stable; this result appears to be folklore and has not appeared anywhere in the literature.  Another failed attempt involves taking the Keisler randomization of a (classical or continuous) simple, unstable theory.  More precisely, either a Keisler randomization is dependent (in which case, if it is simple, then it is stable) or it is not simple;  see \cite{B2}.  In this paper, we will give an example of an essentially continuous theory which is not simple but is rosy (with respect to finitary imaginaries), namely the \emph{Urysohn sphere}, providing the first example of an essentially continuous theory which is unstable and yet possesses a nice notion of independence.

There are many natural ways of defining thorn independence for continuous logic, yielding many notions of rosiness.  The approach which shares the most features with the classical notion is the geometric approach, where one defines thorn-independence to be the independence relation one obtains from the relation of algebraic independence after forcing base monotonicity and extension to hold; this is the approach to thorn independence taken by Adler in \cite{Adler}.  This notion of thorn independence in continuous logic has the new feature that finite character is replaced by countable character, which should not be too surprising to continuous model theorists as the notions of definable and algebraic closure also lose finite character in favor of countable character in the continuous setting.  In order to salvage finite character, we present alternative approaches to thorn independence, yielding notions of rosiness for which we do not know any essentially continuous unstable theories that are rosy.

We now outline the structure of the paper.  In Section 2, we describe some of our conventions concerning continuous logic as well as prove some facts concerning the extensions of definable functions to elementary extensions.  These latter facts have yet to appear in the literature on continuous logic and will only be used in Section 5 in an application of the rosiness of the Urysohn sphere to definable functions.  In Section 3, we introduce the geometric approach to thorn independence and prove some basic results about this notion.  In Section 4, we discuss weak elimination of finitary imaginaries and prove that a continuous real rosy theory which has weak elimination of finitary imaginaries is rosy with respect to finitary imaginaries.  In particular, this shows that a classical real rosy  theory which has weak elimination of imaginaries is rosy, a fact that has yet to appear in the literature on classical rosy theories.  In Section 5, we prove that the Urysohn sphere is real rosy and has weak elimination of finitary imaginaries, whence we conclude that it is rosy with respect to finitary imaginaries.  In Section 6, we introduce other notions of thorn independence and develop properties of these various notions.  In Section 7, we show that if $T$ is a classical theory for which the Keisler randomization $T^R$ of $T$ is \emph{strongly} rosy, then $T$ is rosy; here strongly rosy is one of the alternative notions of rosiness defined in Section 6.

We assume that the reader is familiar with the rudiments of continuous logic; otherwise, they can consult the wonderful survey \cite{BBHU}.  For background information on rosy theories, one can consult \cite{Alf} and \cite{Adler}.  All terminology concerning independence relations will follow \cite{Adler}.    

We would like to thank Ita\"i Ben Yaacov and Ward Henson for helpful discussions involving this work.

\section{Model Theoretic Preliminaries}

In this section, we establish some conventions and notations as well as gather some miscellaneous model-theoretic facts.  First, let us establish a convention concerning formulae.  All formulae will have their variables separated into three parts:  the \emph{object variables}, the \emph{relevant parameter variables}, and the \emph{irrelevant parameter variables}, so a formula has the form $\varphi(x,y,z)$, where $x$ is a multivariable of object variables, $y$ is a multivariable of relevant parameter variables, and $z$ is a multivariable of irrelevant parameter variables.  This distinction will becomes useful in our discussion of thorn-forking, for often only some of the parameter variables are allowed to vary over a type-definable set.  While this distinction is usually glossed over in classical logic, we make a point of discussing it here as the metric on countable tuples is sensitive to the presentation of the tuple.  For ease of exposition, we make the following further convention.  When considering a formula $\varphi(x,y,z)$, we may write $\varphi(x,b)$ to indicate that $b$ is a $y$-tuple being substituted into $\varphi$ for $y$ and we do not care about what parameters are being plugged in for $z$.  When using this convention, if $b'$ is another $y$-tuple, then $\varphi(x,b')$ will denote the formula obtained from $\varphi(x,y,z)$ by substituting $b'$ for $y$ and the \emph{same} tuple for $z$ as in $\varphi(x,b)$.  Finally, let us say that we maintain the conventions of this paragraph for definable predicates as well.

We will use the following metrics on cartesian products.  Suppose that $(M_i,d_i)_{i<\omega}$ are metric spaces.  For two finite tuples $x=(x_0,\ldots,x_n)$ and $y=(y_0,\ldots,y_n)$ from $\prod_{i\leq n} M_i$, we set $d(x,y)=\max_{i\leq n}d_i(x_i,y_i)$.  For two countably infinite tuples $x=(x_i \ | \ i<\omega)$ and $y=(y_i \ | \ i<\omega)$ from $\prod_{i<\omega} M_i$, we set $d(x,y):=\sum_i 2^{-i}d(x_i,y_i)$.  Further suppose that $\mathcal{L}$ is a bounded continuous signature.  Define the signature $\mathcal{L}_\omega$ to be the signature $\mathcal{L}$ together with new sorts for countably infinite products of sorts of $\mathcal{L}$.  We define the metric on these new sorts as above.  We also include projection maps:  if $(S_i \ | \ i<\omega)$ is a countable collection of sorts of $\mathcal{L}$, we add function symbols $\pi_{S,j}:\prod_i S_i \to S_j$ to the language for each $j<\omega$.  Each $\mathcal{L}$-structure expands to an $\mathcal{L}_\omega$-structure in the obvious way.

For any $r\in \r^{>0}$ and any $x\in [0,1]$, we set $r\odot x:=\max(rx,1)$.  Also, for $x,y\in \r^{>0}$, we set $x\dotminus y:=\max(x-y,0)$.

In all but the last section of this paper, $L$ denotes a fixed bounded continuous signature.  For simplicity, let us assume that $L$ is $1$-sorted and the metric $d$ is bounded by $1$.  We also fix a complete $L$-theory $T$ and a monster model $\m$ for $T$.  We let $\kappa(\m)$ denote the saturation cardinal for $\m$ and we say that a parameterset is \emph{small} if it is of cardinality $<\kappa(\m)$.

If $\varphi(x)$ is an $\m$-definable predicate, then we set $\z(\varphi(x))$ to be the \emph{zeroset of $\varphi(x)$}, that is, $\z(\varphi(x))=\{a\in \m_x \ | \ \varphi(a)=0\}$.  Likewise, if $p(x)$ is a type, we write $\z(p(x))=\bigcap \{\z(\varphi(x)) \ | \ ``\varphi(x)=0\text{''}\in p\}$.

Let us briefly recall the $\eq$-construction for continuous logic.  Suppose that $\varphi(x,y)$ is a definable predicate, where $x$ is a finite tuple of object variables and $y$ is a countable tuple of parameter variables.  Then in $\m^{\eq}$, there is a sort $S_\varphi$ whose objects consist of \emph{canonical parameters} of instances of $\varphi$.  Formally, $S_\varphi=\m_y/(d_\varphi=0)$, where $d_\varphi$ is the pseudometric on $\m_y$ given by $$d_\varphi(a,a'):=\sup_x|\varphi(x,a)-\varphi(x,a')|.$$  (Ordinarily, one has to take the completion of $\m_y/(d_\varphi=0)$, but the saturation assumption on $\m$ guarantees that this metric space is already complete.)  As in classical logic, one also adds appropriate projection maps to the language.  For more details on the $\eq$-construction in continuous logic, including axiomatizations of $T^{\eq}$, see Section 5 of \cite{BU}.  In the case when $\varphi(x,y)$ is a \emph{finitary} definable predicate, that is, when $y$ is finite, we say that the elements of $S_\varphi$ are \emph{finitary imaginaries}.  We let $\m^{\feq}$ denote the reduct of $\m^{\eq}$ which retains only sorts of finitary imaginaries.    If $a\in \m^{\eq}$ and $b$ is an element of the equivalence class corresponding to $a$, we write $\pi(b)=a$.  If $A\subseteq \m^{\eq}$ and $B\subseteq \m$, we write $\pi(B)=A$ to indicate the fact that the elements of $B$ are representatives of classes of elements of $A$.  

The remainder of this section will be devoted to understanding extensions to $\m$ of definable functions on small elementary submodels of $\m$; this material will only be used at the end of Section 5.  Suppose that $M$ is a small elementary submodel of $\m$, $A\subseteq M$ is a set of parameters, and $P:M^n\to [0,1]$ is a predicate definable in $M$ over $A$.  Then there exists a unique predicate $Q:\m^n\to [0,1]$ definable in $\m$ over $A$ which has $P$ as its restriction to $M^n$; see \cite{BBHU}, Proposition 9.8.  The predicate $Q$ satisfies the additional property that $(M,P)\preceq (\m,Q)$.  We call $Q$ the \emph{natural extension of $P$ to $\m$}.  Now suppose that $f:M^n\to M$ is $A$-definable, where $A\subseteq M$.  Let $P:M^{n+1}\to [0,1]$ be the $A$-definable predicate $d(f(x),y)$.  Let $Q:\m^{n+1}\to [0,1]$ be the natural extension of $P$ to $\m^{n+1}$.  Then, since $(M,P)\preceq (\m,Q)$, the zeroset of $Q$ defines the graph of a function $g:\m^n\to \m$.  Moreover, $g$ is $A$-definable and extends $f$.  (See \cite{BBHU}, Proposition 9.25)  We call $g$ the \emph{natural extension of $f$ to $\m^n$}.

In Lemma \ref{L:injectiveextension} below, we seek to show that under certain mild saturation assumptions, the natural extension of $f$ to $\m^n$ can preserve some of the properties of $f$.  First, we need two lemmas, the first of which is a generalization of \cite{BBHU}, Proposition 7.14.

\begin{lemma}\label{L:714generalized}
Suppose $M$ is a small elementary submodel of $\m$ and $P,Q:M^n\to [0,1]$ are predicates definable in $M$.  Suppose that either of the following two conditions hold:
\begin{enumerate}
\item [(i)]$M$ is $\omega_1$-saturated, or
\item [(ii)] $M$ is $\omega$-saturated and $P$ is definable over a finite set of parameters from $M$.
\end{enumerate}
Then the following are equivalent:
\begin{enumerate}
\item For all $a\in M^n$, if $P(a)=0$, then $Q(a)=0$.
\item For all $\epsilon>0$, there is $\delta>0$ such that for all $a\in M^n$, $(P(a)\leq \delta\Rightarrow (Q(a)< \epsilon)$.
\item There is an increasing, continuous function $\alpha:[0,1]\to [0,1]$ with $\alpha(0)=0$ such that $Q(a)\leq \alpha(P(a))$ for all $a\in M^n$.
\end{enumerate}
\end{lemma}

\begin{proof}
We only need to prove the direction $(1)\Rightarrow (2)$, as the direction $(2)\Rightarrow (3)$ follows immediately from Proposition 2.10 in \cite{BBHU} and the direction $(3)\Rightarrow (1)$ is trivial.  Suppose that $P(x)$ is the uniform limit of the sequence $(\varphi_m(x) \ | \ m\geq 1)$ and $Q(x)$ is the uniform limit of the sequence $(\psi_m(x) \ | \ m\geq 1)$, where each $\varphi_m(x)$ and $\psi_m(x)$ are formulae with parameters from $M$.  If condition (ii) in the statement of the lemma holds, then we further assume that the parameters from each of the $\varphi_n$'s are contained in some finite subset of $M$.  Moreover, we may assume that $|P(x)-\varphi_m(x)|,|Q(x)-\psi_m(x)|\leq \frac{1}{m}$ for each $m\geq 1$ and each $x\in \m^n$.  Now suppose (2) fails for some $\epsilon>0$.  Then for every $m\geq 1$, there is $a_m\in M^n$ such that $P(a_m)\leq \frac{1}{m}$ and $Q(a_m)\geq \epsilon$.  Let $k\geq 1$ be such that $\epsilon>\frac{3}{k}$.

\noindent \textbf{Claim:}  The collection of conditions $$\Gamma(x):=\{\psi_k(x)\geq \frac{2}{k}\}\cup \{\varphi_m(x)\leq \frac{2}{m} \ | \ m\geq 1\}$$ is finitely satisfiable.

\noindent \textbf{Proof of Claim:}  Consider $m_1,\ldots,m_s\geq 1$.  Set $m':=\max(m_1,\ldots,m_s)$.  Then $\psi_k(a_{m'})\geq Q(a_{m'})-\frac{1}{k}\geq \frac{2}{k}$ and, for each $i\in \{1,\ldots,s\}$, we have $\varphi_{m_i}(a_{m'})\leq P(a_{m'})+\frac{1}{m_i}\leq \frac{1}{m'}+\frac{1}{m_i}\leq\frac{2}{m_i}$.

By the claim and either of assumptions (i) or (ii), we have $a\in M^n$ realizing $\Gamma(x)$.  Then $Q(a)\geq \psi_k(a)-\frac{1}{k}\geq \frac{1}{k}$.  Also, $P(a)\leq \varphi_m(x)+\frac{1}{m}\leq \frac{3}{m}$ for all $m\geq 1$, whence $P(a)=0$.  Thus, (2) fails, finishing the proof of the lemma.
\end{proof}

The import of the above lemma is the following.  Working in the notation of the lemma, suppose that $P$ and $Q$ satisfy (1) and either (i) or (ii) holds.  Suppose $P'$ and $Q'$ denote the natural extensions of $P$ and $Q$ to $\m^n$.  Then it follows that, for all $a\in \m^n$, $P'(a)=0\Rightarrow Q'(a)=0$.  This is because $(1)$ is equivalent to $(3)$, which can be expressed by a formula in the signature of the structure $(M,P,Q)$.  Since $(M,P,Q)\preceq (\m,P',Q')$, we have that (3) holds with $P'$ and $Q'$ replacing $P$ and $Q$.  This in turn implies that $(1)$ holds with $P'$ and $Q'$ replacing $P$ and $Q$.

\begin{lemma}\label{L:natext}
Suppose $M$ is a small $\omega$-saturated elementary submodel of $\m$ and $A\subseteq M$ is countable.  Let $f:M^n\to M$ be an $A$-definable function and let $g:\m^n\to \m$ be the natural extension of $f$ to $\m^n$.  Let $R:M^{2n}\to[0,1]$ be the predicate defined by $R(a,b)=d(f(a),f(b))$ for all $a,b\in M^n$, which is definable in $M$ over $A$.  Let $S:\m^{2n}\to[0,1]$ be the natural extension of $R$ to $\m^{2n}$.  Then $S(a,b)=d(g(a),g(b))$ for all $a,b\in \m^n$.
\end{lemma}

\begin{proof}
Let $(\varphi_m(x,y) \ | \ m\geq 1)$ be a sequence of formulae with parameters from $A$ converging uniformly to the predicate $P(x,y):=d(f(x),y)$.  Further assume that $|P(x,y)-\varphi_m(x,y)|\leq \frac{1}{m}$ for all $(x,y)\in M^{n+1}$ and all $m\geq 1$.  Note that if $(a,c),(b,d)\in M^{n+1}$ are such that $\varphi_m(a,c),\varphi_m(b,d)\leq \frac{1}{m}$, then $|R(a,b)-d(c,d)|\leq \frac{4}{m}$.  By Lemma \ref{L:714generalized}, we have that for all $(a,c),(b,d)\in \m^{n+1}$, if $\varphi_m(a,c),\varphi_m(b,d)\leq \frac{1}{m}$, then $|S(a,b)-d(c,d)|\leq \frac{4}{m}$.  It remains to show that $\varphi_m(a,g(a)),\varphi_m(b,g(b))\leq \frac{1}{m}$ for each $m\geq 1$.  However, this follows from the fact that $Q(a,g(a))=Q(b,g(b))=0$ and $|Q(x,y)-\varphi_m(x,y)|\leq \frac{1}{m}$ for all $(x,y)\in \m^{n+1}$ and all $m\geq 1$.
\end{proof}

\begin{lemma}\label{L:injectiveextension}
Suppose $M$ is a small elementary submodel of $\m$ and $A\subseteq M$ is countable.  Let $f:M^n\to M$ be an $A$-definable function and let $g:\m^n\to \m$ be the natural extension of $f$ to $\m^n$.
\begin{enumerate}
\item Suppose $M$ is $\omega$-saturated and $f$ is an isometric embedding.  Then $g$ is also an isometric embedding.
\item Suppose that either:
\begin{enumerate}
\item $M$ is $\omega_1$-saturated, or
\item $M$ is $\omega$-saturated and $A$ is finite.
\end{enumerate}
Further suppose that $f$ is injective.  Then $g$ is injective.
\end{enumerate}
\end{lemma}

\begin{proof}
Define the predicates $R$ and $S$ as in Lemma \ref{L:natext}.

(1)  Fix $\epsilon>0$.  Then for all $a,b\in M^n$, we have $$|d(a,b)-\epsilon|=0 \Rightarrow |R(a,b)-\epsilon|=0.$$  By Lemma \ref{L:714generalized}, we have, for all $a,b\in \m^n$, $$|d(a,b)-\epsilon|=0 \Rightarrow |S(a,b)-\epsilon|=0.$$  It follows that $g$ is an isometric embedding.  

(2)  Since $f$ is injective, we know that, for all $a,b\in M^n$, if $R(a,b)=0$, then $d(a,b)=0$.  By Lemma \ref{L:714generalized}, for all $a,b\in \m^n$, we have $$S(a,b)=0 \Rightarrow d(a,b)=0.$$  It follows that $g$ is injective.
\end{proof}

\section{Basic Properties of Thorn-Forking}

In classical logic, there are two ways of defining thorn-independence:  a ``geometric'' definition and a ``formula'' definition.  Since the geometric definition immediately makes sense in continuous logic, we shall use it to define thorn-independence for continuous logic.  Afterwards, we explain an equivalent formula definition.  

\

\noindent \textbf{Notation:}  For any sets $X,Y$ with $X\subseteq Y$, we set 
$$[X,Y]:=\{Z \ | \ X\subseteq Z \subseteq Y\}.$$

We borrow the following definitions from Adler \cite{Adler} to define thorn-independence in the continuous setting.

\begin{df}
Let $A,B,C$ be small subsets of $\m^{\eq}$.
\begin{enumerate}
\item We write $A\ind[M]_CB$ if and only if for any $C'\in [C,\acl(BC)]$, we have $\acl(AC')\cap \acl(BC')=\acl(C')$.
\item We write $A\thind_CB$ if and only if for any $B'\supseteq B$, there is $A'\equiv_{BC}A$ with $A'\ind[M]_CB'$.
\end{enumerate}
\end{df}

For the sake of the reader who has not seen the above definitions, let us take a moment to motivate them.  One of the most natural ternary relations amongst small subsets of the monster model is the relation of ``algebraic independence,'' namely $A$ is algebraically independent from $B$ over $C$ if $\acl(AC)\cap \acl(BC)= \acl(C)$.  This relation is not always an independence relation as it may fail to satisfy base monotonicity.  The relation $\ind[M] \ $ is an attempt to force base monotonicity to hold.  However, $\ind[M] \ $ may fail to satisfy extension, that is, nonforking extensions to supersets may not exist.  Thus, $\thind$ is introduced in order to force extension to hold.

In \cite{Adler}, Adler shows that the relation $\ind[M]\ $ (for classical theories) satisfies:  invariance, monotonicity, base monotonicity, transitivity, normality, and anti-reflexivity.  These properties persist for $\ind[M] \ $ in continuous logic.  In classical logic, $\ind[M]\ $ also satisfies finite character, whereas in continuous logic, $\ind[M]\ $ satisfies \emph{countable character}:  If $A_0\ind[M]_CB$ for every countable $A_0\subseteq A$, then $A\ind[M]_CB$.  The proof of this is the same as in \cite{Adler}, using the fact that if $b\in \acl(A)$, then there is a countable $A_0\subseteq A$ such that $b\in \acl(A_0)$.  Adler shows that $\thind$ satisfies invariance, monotonicity, base monotonicity, transitivity, normality, and anti-reflexivity; these properties remain true for $\thind$ in continuous logic.  In \cite{Adler}, it is also shown that $\thind$ has finite character provided $\ind[M]\ $ has finite character and $\thind$ has local character.  Using Morley sequences indexed by $\omega_1$ instead of $\omega$, Adler's arguments show that, in continuous logic, $\thind$ satisfies countable character provided $\ind[M]\ $ satisfies countable character and $\thind$ satisfies local character.

In Remark 4.1 of \cite{Adler}, it is shown that, in classical logic, if $\ind$ is any strict independence relation, then $\ind \Rightarrow \thind$.  This proof does not use finite character and so remains true in continuous logic.  Let us summarize this discussion with the following theorem, where a \emph{countable independence relation} is an independence relation satisfying countable character instead of finite character.

\begin{thm}
The relation $\thind$ is a strict countable independence relation if and only if it has local character if and only if there is a strict countable independence relation at all.  If $\thind$ is a strict countable independence relation, then it is the weakest.
\end{thm}

\begin{df}
$T$ is said to be \textbf{rosy} if $\thind$ satisfies local character.  $T$ is said to be \textbf{real rosy} (resp. \textbf{rosy with respect to finitary imaginaries}) if $\thind$ satisfies local character when restricted to the real sorts (resp. sorts of finitary imaginaries). 
\end{df}

\begin{cor}
Simple continuous theories are rosy.
\end{cor}

\begin{proof}
If $T$ is simple, then dividing independence is a strict independence relation for $T^{\eq}$.
\end{proof}

\begin{cor}
If $T$ is a classical theory viewed as a continuous theory, then $T$ is rosy as a classical theory if and only if $T$ is rosy with respect to finitary imaginaries as a continuous theory.
\end{cor}

\begin{proof}
This follows from the fact that $\m^{\eq}$ (in the classical sense) is the same as $\m^{\feq}$ and, for $A\subseteq \m^{\feq}$, the algebraic closure of $A$ is the same in either structure.  
\end{proof}

We next seek to provide a formula definition for thorn-independence.  First, we will need some definitions.

\begin{df}
Suppose that $B$ is a small subset of $\m^{\eq}$ and $c$ is a countable tuple from $\m^{\eq}$.
\begin{enumerate}
\item We let $\i(c/B)$ denote the set of $B$-indiscernible sequences of realizations of $\tp(c/B)$.
\item If $I\in \i(c/B)$, let $d(I):=d(c',c'')$ for any $c',c''\in I$.
\item We let $\chi(c/B):=\max\{ d(I) \ | \ I \in \i(c/B)\}$.
\end{enumerate}
\end{df}

\begin{rmks}
Suppose that $B$ and $D$ are small subsets of $\m^{\eq}$ and $c$ is a countable tuple from $\m^{\eq}$. 
\begin{enumerate}
\item If $B\subseteq D$, then $\chi(c/D)\leq \chi(c/B)$.
\item Lemma 4.9 in \cite{BU} shows that $\chi(c/B)=0$ if and only if $\tp(c/B)$ is algebraic.
\item Since the metric on countably infinite tuples is sensitive to the enumeration of the tuple, it is possible that if $c$ is countably infinite and $c'$ is a rearrangement of $c$, then $\chi(c/B)$ may not equal $\chi(c'/B)$.  However, $\chi(c/B)=0$ if and only if $\chi(c'/B)=0$ as a tuple is algebraic over $B$ if and only if each component of the tuple is algebraic over $B$.
\end{enumerate}
\end{rmks}

\begin{df}\label{D:thind}

Suppose $\varphi(x,y)$ is a formula, $\epsilon >0$, $c$ is a countable tuple from $\m^{\eq}$, and $B$ is a small subset of $\m^{\eq}$.

\begin{enumerate}
\item We say that $\varphi(x,c)$ \textbf{strongly $\epsilon$-$k$-divides over $B$} if:
\begin{itemize}
\item $\epsilon\leq \chi(c/B)$, and 
\item whenever $c_1,\ldots,c_k\models \tp(c/B)$ satisfy $d(c_i,c_j)\geq \epsilon$ for all \\$1\leq i<j\leq k$, we have $$\models \inf_x \max_{1\leq i\leq k}\varphi(x,c_i)=1.$$
\end{itemize}
\item We say that $\varphi(x,c)$ \textbf{strongly $\epsilon$-divides over $B$} if it strongly $\epsilon$-$k$-divides over $B$ for some $k\geq 1$.
\item We say that $\varphi(x,c)$ \textbf{strongly $\epsilon$-$k$-divides over $B$ in the na\"ive sense} if:
\begin{itemize}
\item $\epsilon\leq \chi(c/B)$, and 
\item whenever $c_1,\cdots,c_k\models \tp(c/B)$ satisfy $d(c_i,c_j)\geq \epsilon$ for all \\$1\leq i<j\leq k$, we have, for every $a\in \m_x$, that $$\max_{1\leq i\leq k}\varphi(a,c_i)>0.$$   
\end{itemize} 
We say that $\varphi(x,c)$ \textbf{strongly $\epsilon$-divides over $B$ in the na\"ive sense} if it strongly $\epsilon$-$k$-divides over $B$ in the na\"ive sense for some $k\geq 1$.
\end{enumerate}
\end{df}

Using our conventions from above, when saying that $\varphi(x,c,d)$ strongly $\epsilon$-divides over $B$, we only consider $B$-conjugates of $c$ which are $\epsilon$-apart; $d$ must remain fixed.  The next proposition is the key link between the geometric and formula definitions of thorn-independence.

\begin{prop}\label{L:aclstrongdivide}
Let $A$ and $C$ be small parametersets from $\m^{\eq}$ and let $b$ be a countable tuple from $\m^{\eq}$.  Then the following are equivalent:
\begin{enumerate}
\item $b\in \acl(AC)\setminus \acl(C)$;
\item $b\notin \acl(C)$ and for every $\epsilon $ with $0<\epsilon\leq \chi(b/C)$, there is a formula $\varphi_\epsilon(x,b)$ such that the condition ``$\varphi_\epsilon(x,b)=0$'' is in $\tp(A/bC)$ and such that $\varphi_\epsilon(x,b)$ strongly $\epsilon$-divides over $C$  in the na\"ive sense.
\item $b\notin \acl(C)$ and for every $\epsilon $ with $0<\epsilon\leq \chi(b/C)$, there is a formula $\varphi_\epsilon(x,b)$ such that the condition ``$\varphi_\epsilon(x,b)$'' is in $\tp(A/bC)$ and such that $\varphi_\epsilon(x,b)$ strongly $\epsilon$-divides over $C$.
\end{enumerate}
\end{prop}

\begin{proof}
$(1)\Rightarrow (2)$:  Suppose that (2) fails.  If $b\in \acl(C)$, then (1) fails.  Assume that $b\notin \acl(C)$.  We aim to show that $b\notin \acl(AC)$.  We argue as in the proof of Lemma 2.1.3(4) in \cite{Alf}.  By assumption, there is $\epsilon$ with $0<\epsilon \leq \chi(b/C)$ such that $\varphi(x,b)$ doesn't strongly $\epsilon$-divide over $C$ in the na\"ive sense for any formula $\varphi(x,b)$ such that the condition ``$\varphi(x,b)=0$'' is in $\tp(A/bC)$.  Let $p(X,y):=\tp(A,b/C)$ and $q(y):=\tp(b/C)$.

\noindent \textbf{Claim:}  The set of $(L^{\eq})_\omega$-conditions $$\Gamma(X,(y_i)_{i<\omega}):=\bigcup_{i<\omega}p(X,y_i)\cup \bigcup_{i<\omega} q(y_i) \cup \{d(y_i,y_j)\geq \epsilon \ | \ i<j<\omega\}$$ is satisfiable.

It is enough to prove that, for any $\varphi(x,y)$ for which the condition ``$\varphi(x,b)=0$'' is in $p(X,b)$ and any $n<\omega$, we have $$\{\max_{i\leq n}\varphi(x,y_i)=0\}\cup \bigcup_{i\leq n}q(y_i) \cup \{d(y_i,y_j)\geq \epsilon \ | \ i<j\leq n\}$$ is satisfiable.  Since $b\notin \acl(C)$ and $\varphi(x,b)$ doesn't strongly $\epsilon$-divide over $C$ in the na\"ive sense, we have $b_0,\ldots,b_n\models q$ such that $d(b_i,b_j)\geq \epsilon$ for all $i<j\leq n$ and $\max_{i\leq n}\varphi(c,b_i)=0$ for some $c$.  This finishes the proof of the claim.

Let $(A,(b_i)_{i<\omega})$ realize $\Gamma(X,(y_i)_{i<\omega})$.  Since $A'b_0\equiv_C Ab$, we may assume $A'b_0=Ab$.  It then follows that $b'\models \tp(b/AC)$ for all $i<\omega$.  Since $(b_i)_{i<\omega}$ can contain no convergent subsequence, the set of realizations of $\tp(b/AC)$ in $\m^{\eq}$ cannot be compact, whence $b\notin \acl(AC)$.

$(2)\Rightarrow (3)$:  Suppose (2) holds and fix $\epsilon$ with $0<\epsilon\leq \chi(b/C)$.  Suppose $\varphi_\epsilon(x,b)$ is a formula such that the condition ``$\varphi_\epsilon(x,b)=0$'' is in $\tp(A/bC)$ and such that $\varphi_\epsilon(x,b)$ strongly $\epsilon$-$k$-divides over $C$ in the na\"ive sense.  By compactness, we can find $r\in (0,1]$ such that $\inf_x\max_{i<k}\varphi_\epsilon(x,b_i)\geq c$ for all $b_0,\ldots,b_{k-1}\models \tp(b/C)$ with $d(b_i,b_j)\geq \epsilon$ for all $i<j<k$.  Let $\varphi_{\epsilon}':=\frac{1}{r}\odot \varphi_\epsilon$.  Then the condition ``$\varphi_\epsilon'(x,b)=0$'' is in $\tp(A/bC)$ and $\varphi_\epsilon'(x,b)$ strongly $\epsilon$-$k$-divides over $C$.

$(3)\Rightarrow (1)$:  Suppose that (3) holds and yet $b\notin \acl(AC)$, i.e. the set $X$ of realizations of $\tp(b/AC)$ in $\m^{\eq}$ is not compact.  Note that $X$ is closed, and hence complete.  It follows that $X$ is not totally bounded, i.e. there is $\epsilon>0$ such that $X$ cannot be covered by finitely many balls of radius $\epsilon$.  Without loss of generality, we may assume that $\epsilon\leq \chi(b/C)$.  Let $\varphi_\epsilon(x,b)$ be such that the condition ``$\varphi_\epsilon(x,b)=0$'' is in $\tp(A/bC)$ and such that $\varphi_\epsilon(x,b)$ strongly $\epsilon$-$k$-divides over $C$.  Choose $b_1,\ldots,b_k\in X$ with $d(b_i,b_j)\geq \epsilon$.  Then $\varphi_\epsilon(A,b_i)=0$ for each $i\in \{1,\ldots,k\}$, contradicting strong $\epsilon$-$k$-dividing.
\end{proof}

Motivated by the above proposition, we make the following definitions.

\begin{df}
Let $A,B,C$ be small subsets of $\m^{\eq}$.
\begin{enumerate}
\item If $b$ is a countable tuple, then $\tp(A/bC)$ \textbf{strongly divides over $C$} if $b\in \acl(AC)\setminus \acl(C)$.
\item $\tp(A/BC)$ \textbf{strongly divides over $C$} if $\tp(A/bC)$ strongly divides over $C$ for some countable $b\subseteq B$.
\item If $b$ is a countable tuple, then $\tp(A/bC)$ \textbf{thorn-divides over $C$} if there is a $D\supseteq C$ such that $b\notin \acl(D)$ and such that, for every $\epsilon $ with $0<\epsilon\leq \chi(b/D)$, there is a formula $\varphi_\epsilon(x,b)$ such that the condition ``$\varphi_\epsilon(x,b)=0$'' is in $\tp(A/bC)$ and $\varphi_\epsilon(x,b)$ strongly $\epsilon$-divides over $D$. 
\item $\tp(A/BC)$ \textbf{thorn-divides over $C$} if $\tp(A/bC)$ thorn-divides over $C$ for some countable $b\subseteq B$.
\item $\tp(A/BC)$ \textbf{thorn-forks over $C$} if there is $E\supseteq BC$ such that every extension of $\tp(A/BC)$ to $E$ thorn-divides over $C$.  
\end{enumerate}
\end{df}

The following is the continuous analog of Theorem 3.3 in \cite{Alfclif}.
\begin{thm}\label{T:weakest}
Suppose $\ind[I]$ is an automorphism-invariant ternary relation on small subsets of $\m^{\eq}$ satisfying, for all small $A,B,C,D$:
\begin{enumerate}
\item for all countable $b$, if $b\in \acl(AC)\setminus \acl(C)$, then $A\nind[I]_Cb$;
\item if $A\ind[I]_B D$ and $B\subseteq C \subseteq D$, then $A\ind[I]_C D$ and $A\ind[I]_BC$;
\item if $A\ind[I]_C B$ and $BC\subseteq D$, then there is $A'\equiv_{BC}A $ such that $A'\ind[I]_CD$.
\item if $A\ind[I]_CBC$, then $A\ind[I]_CB$.
\end{enumerate}
Then for all $A,B,C$, if $A\ind[I]_CB$, then $\tp(A/BC)$ does not thorn-fork over $C$.
\end{thm}

\begin{proof}
Let us first show that thorn-dividing implies $I$-dependence.  Suppose that $\tp(A/BC)$ thorn-divides over $C$ but, towards a contradiction, $A\ind[I]_CB$.  Let $b\subseteq B$ be countable so that $\tp(A/bC)$ thorn-divides over $C$.  We then have $D\supseteq C$ such that $b\notin \acl(D)$ and for every $\epsilon \in (0,\chi(b/D)]$, there is a formula $\varphi_\epsilon(x,b)$ such that the condition ``$\varphi_\epsilon(x,b)=0$'' is in $\tp(A/bC)$ and $\varphi_\epsilon(x,b)$ strongly $\epsilon$-divides over $D$.

\noindent \textbf{Claim:}  $D$ can be chosen so that $A\ind[I]_CbD$.

By (3), we have $A\ind[I]_CBC$.  By (2), we have $A\ind[I]_CbC$.  By (4), we have $A\ind[I]_Cb$.  By (2), there is $a'\equiv_{bC}a$ such that $a'\ind[I]_CbD$.  Take $\sigma\in \Aut(\m^{\eq}|bC)$ such that $\sigma(a')=a$.  Since $a\ind[I]_Cb\sigma(D)$, $\sigma(D)\supseteq C$, and $\chi(b/\sigma(D))=\chi(b/D)$, we have $\varphi_\epsilon(x,b)$ still strongly $\epsilon$-divides over $\sigma(D)$.  This finishes the proof of the claim.

By the Claim and (2), we have $A\ind[I]_DbD$, and by (4), we have $A\ind[I]_Db$.  However, by Proposition \ref{L:aclstrongdivide}, we have $b\in \acl(AD)\setminus \acl(D)$, so by (1), we have $A\nind[I]_Db$, a contradiction.

Now suppose that $A\ind[I]_CB$.  We wish to show that $\tp(A/BC)$ does not thorn-fork over $C$.  Fix $E\supseteq BC$.  By (2), we have $A'\equiv_{BC}A$ such that $A'\ind[I]_CE$.  By the first part of the proof, we have $\tp(A'/E)$ does not thorn-divide over $C$.  Since $\tp(A/BC)$ has an extension to every superset of $BC$ which does not thorn-divide over $C$, it follows that $\tp(A/BC)$ does not thorn-fork over $C$.    
\end{proof}

In establishing the equivalence of the geometric and formula definitions of thorn-independence, the following technical lemma will be useful.

\begin{lemma}\label{L:countable}
Suppose that $b$ is countable and $\tp(A/bC)$ thorn-divides over $C$, witnessed by $D\supseteq C$.  Then we can find $D'\in [C,D]$ witnessing that $\tp(A/bC)$ thorn-divides over $C$ and satisfying $|D'\setminus C|\leq \aleph_0$.
\end{lemma}

\begin{proof}
Fix $\epsilon \in (0,\chi(b/D)]$.  Choose $\varphi_\epsilon(x,b)$ such that the condition ``$\varphi_\epsilon(x,b)=0$'' is in $\tp(A/bC)$ and $\varphi_\epsilon(x,b)$ strongly $\epsilon$-$k$-divides over $D$ for some $k\geq 1$.  By compactness, there is a finite $d_\epsilon\subseteq D$ and a formula $\psi(y,d_\epsilon)$ such that $\psi(b,d_\epsilon)=0$ and whenever $b_0,\ldots,b_{k-1}$ are such that $\psi(b_i,d_\epsilon)=0$ and $d(b_i,b_j)\geq \epsilon$ for all $i<j<k$, we have $\inf_x \max_{1\leq i \leq k_\epsilon} 2\odot \varphi_\epsilon(x,b_i)=1$.  Let $D':=C\cup \bigcup\{d_\epsilon \ | \ \epsilon \in (0,\chi(c/D)]\cap \mathbb{Q}\}$.  It follows that this $D'$ has the desired property.    
\end{proof}

A version of the following proposition appears in \cite{Adler} for classical theories.

\begin{prop}\label{L:equivalence}
Let $A$ and $C$ be arbitrary small subsets of $\m^{\eq}$.  Let $M$ be a small elementary submodel of $\m^{\eq}$ such that $C\subseteq M$ and $M$ is $(|T|+|C|)^+$-saturated.  Then the following are equivalent:
\begin{enumerate}
\item $A\ind[M]_C M$;
\item for all $C'\in [C,M]$, we have $\acl(AC')\cap M=\acl(C')$;
\item for all $C'\in [C,M]$, we have $\tp(A/M)$ does not strongly divide over $C'$;
\item $\tp(A/M)$ does not thorn-divide over $C$.
\end{enumerate}
\end{prop}

\begin{proof}
$(1)\Leftrightarrow (2)$ is immediate.  

$(2)\Rightarrow (3)$:  Suppose there is $C'\in [C,M]$ such that $\tp(A/M)$ strongly divides over $C'$.  Choose $b\subseteq M$ countable such that $\tp(A/bC')$ strongly divides over $C'$, i.e. $b\in \acl(AC')\setminus \acl(C')$.  Writing $b=(b_i)_{i<\omega}$, by Proposition 2.8(2) of \cite{HenTell}, there is $i<\omega$ such that $b_i\in \acl(AC')\setminus \acl(C')$, contradicting (2).  

$(3)\Rightarrow (2)$:  Suppose there is $C'\in [C,M]$ and $b\in M$ such that $b\in \acl(AC')\setminus \acl(C')$.  Then $\tp(A/bC')$ strongly divides over $C'$, whence $\tp(A/M)$ strongly divides over $C'$.

$(3)\Rightarrow(4)$:  Suppose that $\tp(A/M)$ thorn-divides over $C$.  Choose $b\subseteq M$ countable such that $\tp(A/bC)$ thorn-divides over $C$.  By Lemma \ref{L:countable}, we can find a countable $d\subseteq \m^{\eq}$ such that $b\notin \acl(Cd)$ and for every $\epsilon \in (0,\chi(b/Cd)]$, there is a formula $\varphi_\epsilon(x,b)$ such that the condition ``$\varphi_\epsilon(x,b)=0$'' is in $\tp(A/bC)$ and $\varphi_\epsilon(x,b)$ strongly $\epsilon$-divides over $Cd$.  Let $d'\subseteq M$ be such that $d'\equiv_{bC}d$; this is possible by the saturation assumption on $M$.  Now notice that $\tp(A/bCd')$ strongly divides over $Cd'$, whence $\tp(A/M)$ strongly divides over $Cd'$, contradicting (3).

$(4)\Rightarrow (3)$:  Suppose that there is $C'\in [C,M]$ such that $\tp(A/M)$ strongly divides over $C'$.  Let $b\subseteq M$ be countable such that $\tp(A/bC')$ strongly divides over $C'$.  Arguing as in Lemma \ref{L:countable}, we may find countable $d\subseteq C'$ such that $\tp(A/bCd)$ strongly divides over $Cd$.  We now show that $\tp(A/bdC)$ thorn-divides over $C$, whence $\tp(A/M)$ thorn-divides over $C$, finishing the proof of the proposition.  Since the metric on countably infinite tuples is sensitive to the enumeration of the tuple, we must specify the enumeration of $bd$.  We fix the enumeration $bd=(b_0,d_0,b_1,d_1,\ldots)$.  Notice that if $b'd,b''d\models \tp(bd/Cd)$, then $d(b'd,b''d)\leq d(b',b'')$.  In particular, this shows that $\chi(bd/Cd)\leq \chi(b/Cd)$.  Note also that $bd\notin \acl(Cd)$ as $b\notin \acl(Cd)$.  Fix $\epsilon \in (0,\chi(bd/Cd)]$.  Let $\varphi(x,b)$ be a formula such that the condition ``$\varphi(x,b)=0$'' is in $\tp(A/bCd)$ and such that $\varphi(x,b)$ strongly $\epsilon$-$k$-divides over $Cd$ for some $k\geq 1$; this is possible since $\tp(A/bCd)$ strongly divides over $Cd$.  Now suppose $b_0d,\ldots,b_{k-1}d\models \tp(bd/Cd)$ are such that $d(b_id,b_jd)\geq \epsilon$ for all $i<j<k$.  Then $d(b_i,b_j)\geq \epsilon$ for all $i<j<k$, whence $\inf_x\max_{i<k}\varphi(x,b_i,d)=1$.  Thus, $Cd$ witnesses that $\tp(a/bdC)$ thorn-divides over $C$.    
\end{proof}

\begin{cor}
For all small $A,B,C\subseteq \m^{\eq}$, we have $A\thind_CB$ if and only if $\tp(A/BC)$ does not thorn-fork over $C$. 
\end{cor}

\begin{proof}
First suppose that $\tp(A/BC)$ thorn-forks over $C$.  Let $E\supseteq BC$ be such that every extension of $\tp(A/BC)$ to $E$ thorn-divides over $C$.  Let $M\supseteq E$ be a small elementary submodel of $\m^{\eq}$ which is $(|T|+|C|)^+$-saturated.  Then every extension of $\tp(A/BC)$ to $M$ thorn-divides over $C$.  By Proposition \ref{L:equivalence}, we see that $A'\nind[M]_CM$ for every $A'\models \tp(A/BC)$, whence $A\nthind_CB$.

Now suppose that $A\nthind_CB$.  Let $E\supseteq BC$ be such that $A'\nind[M]_CE$ for every $A'\models \tp(A/BC)$.  Let $M\supseteq E$ be a small elementary submodel of $\m^{\eq}$ which is $(|T|+|C|)^+$-saturated.  Then by monotonicity of $\ind[M]$, we have $A'\nind[M]_C M$ for every $A'\models \tp(A/BC)$.  By Proposition \ref{L:equivalence}, we see that every extension of $\tp(A/BC)$ to $M$ thorn-divides over $C$.  It follows that $\tp(A/BC)$ thorn-forks over $C$.
\end{proof}

One can define what it means for a definable predicate $\Phi(x,b)$ to strongly $\epsilon$-$k$-divide over a parameterset just as in the case of formulae.  A priori, it appears that we may get a different notion of thorn-forking if we allowed definable predicates to witness strong dividing.  However, this is not the case, as we now explain.  Suppose $A$ and $C$ are small subsets of  $\m^{\eq}$ and $b$ is a countable tuple from $\m^{\eq}$.  Say that $\tp(A/bC)$ thorn$^*$-divides over $C$ if there is a small $D\supseteq C$ such that $b\notin \acl(D)$ and for every $\epsilon$ with $0<\epsilon\leq \chi(b/D)$, there is a \emph{definable predicate} $\Phi_\epsilon(x,b)$ with parameters from $Cb$ such that $\Phi_\epsilon(A,b)=0$ and $\Phi_\epsilon(x,b)$ strongly $\epsilon$-divides over $D$.  For small $A,B,C$, one defines what it means for $\tp(A/BC)$ to thorn$^*$-divide over $C$ and thorn$^*$-fork over $C$ in the obvious ways.

\begin{lemma}\label{L:thornstar}
For small $A,B,C$, $\tp(A/BC)$ thorn$^*$-divides (-forks) over $C$ if and only if it thorn-divides (-forks) over $C$.
\end{lemma}

\begin{proof}
The $(\Leftarrow)$ direction is immediate.  For the $(\Rightarrow)$ direction, suppose $\tp(A/BC)$ thorn$^*$-divides over $C$.  Let $b\subseteq B$ be a countable tuple such that $\tp(A/bC)$ thorn$^*$-divides over $C$, witnessed by $D\supseteq C$.  Fix $\epsilon \in (0,\chi(b/D)]$.  Let $\Phi_\epsilon(x,b)$ be a definable predicate with parameters from $Cb$ such that $\Phi_\epsilon(A,b)=0$ and $\Phi_\epsilon(x,b)$ strongly $\epsilon$-divides over $D$.  Let $\tilde{\varphi}_\epsilon(x,b)$ be an $L(Cb)$-formula such that $\sup_x|\Phi_\epsilon(x,b)-\tilde{\varphi}_\epsilon(x,b)|\leq \frac{1}{4}$.  Let $\varphi_\epsilon(x,b):=4\odot(\tilde{\varphi}_\epsilon(x,b)\dotminus \frac{1}{2})$.  Note that $\tilde{\varphi}_\epsilon(A,b)\leq \frac{1}{4}$, whence $\varphi_\epsilon(A,b)=0$.  It remains to show that $\varphi_\epsilon(x,b)$ strongly $\epsilon$-divides over $D$.  Suppose $\Phi_\epsilon(x,b)$ strongly $\epsilon$-$k$-divides over $D$.  Let $b_1,\ldots,b_k\models \tp(b/D)$ be $\epsilon$-apart.  Fix $e\in (\m^{\eq})_x$.  Choose $i\in \{1,\ldots,k\}$ such that $\Phi_\epsilon(e,b_i)=1$.  For this $i$, we have $\tilde{\varphi}_\epsilon(e,b_i)\geq \frac{3}{4}$, so $\tilde{\varphi}_\epsilon(e,b_i)\dotminus \frac{1}{2}\geq \frac{1}{4}$, whence $\varphi_\epsilon(e,b_i)=1$.
\end{proof}

Let us end this section with the definition of superrosiness, which is meant to mimic the definition of supersimplicity for continuous logic.

\begin{df}
Suppose that $T$ is rosy.  Then we say that $T$ is \textbf{superrosy} if for any finite tuple $a$ from $\m^{\eq}$, any small $B\subseteq \m^{\eq}$, and any $\epsilon>0$, there is a finite tuple $c$ which is similar to $a$ and a finite $B_0\subseteq B$ such that $d(a,c)<\epsilon$ and $c\thind_{B_0}B$.
\end{df}

\section{Weak Elimination of Finitary Imaginaries}

In this section, we discuss what it means for a continuous theory to weakly eliminate finitary imaginaries.  We then show that a real rosy (continuous) theory which weakly eliminates finitary imaginaries is rosy with respect to finitary imaginaries.  In fact, our proof will show that in classical logic, a real rosy theory which admits weak elimination of imaginaries is rosy, which is a fact that, to our knowledge, has not yet appeared in the literature on classical rosy theories.

The following lemma is the continuous analog of the discussion on weak elimination of imaginaries from \cite{Poizat}, pages 321-323.

\begin{lemma}\label{L:WEFI}
The following conditions are equivalent:
\begin{enumerate}
\item For every finitary definable predicate $\varphi(x,a)$ with real parameters, there is a finite tuple $c$ from $\m$ such that:
\begin{itemize}
\item $\varphi(x,a)$ is a $c$-definable predicate, and 
\item if $B$ is a real parameterset for which $\varphi(x,a)$ is also a $B$-definable predicate, then $c\in \acl(B)$.
\end{itemize}
\item For every finitary definable predicate $\varphi(x,a)$ with real parameters, there is a finite tuple $c$ from $\m$ such that:
\begin{itemize}
\item $\varphi(x,a)$ is a $c$-definable predicate, and 
\item if $d$ is a finite tuple from $\m$ for which $\varphi(x,a)$ is also a $d$-definable predicate, then $c\in \acl(d)$.
\end{itemize}
\item For every finitary definable predicate $\varphi(x,a)$ with real parameters, there is a definable predicate $P(x,c)$, $c$ a finite tuple from $\m$, such that $\varphi(x,a)\equiv P(x,c)$ and the set $$\{c' \ | \ c'\equiv c \text{ and }\varphi(x,a)\equiv P(x,c')\}$$ is compact.
\item For every finitary imaginary $e\in \m^{feq}$, there is a finite tuple $c$ from $\m$ such that $e\in \dcl(c)$ and $c\in \acl(e)$. 
\end{enumerate}
\end{lemma}

\begin{proof}
$(1)\Rightarrow (2)$ is trivial.  $(2)\Rightarrow (3)$:  Fix a finitary definable predicate $\varphi(x,a)$ and let $c$ be as in (2) for $\varphi(x,a)$.  Let $P(x,c)$ be a $c$-definable predicate for which $\varphi(x,a)\equiv P(x,c)$.  We claim that this $P(x,c)$ is as desired.  Set $$X:=\{c' \ | \ c'\equiv c \text{ and }\varphi(x,a)\equiv P(x,c')\}.$$  Let $p(c):=\tp(c/\emptyset)$.  Then $X=p(\m)\cap \z(\sup_x|\varphi(x,a)-P(x,z)|)$, whence $X$ is closed and hence complete.  Suppose that $X$ is not compact.  It follows that $X$ is not totally bounded.  Choose $\epsilon>0$ such that $X$ cannot be covered by finitely many balls of radius $\epsilon$.  By the Compactness Theorem, it follows that $X$ cannot be covered by a small number of balls of radius $\epsilon$.  Since $\acl(c)$ is the union of a small number of sets, each of which can be covered by finitely many balls of radius $\epsilon$, it follows that $X\nsubseteq \acl(c)$.  Let $c'\in X\setminus \acl(c)$.  Take $\sigma\in \Aut(\m)$ such that $\sigma(c')=c$.  Set $c'':=\sigma(c)$.  Since $P(x,c)\equiv P(x,c')$, we have $P(x,c'')\equiv P(x,c)$, whence $\varphi(x,a)$ is defined over $c''$.  It follows that $c\in \acl(c'')$.  However, applying $\sigma^{-1}$, we get $c'\in \acl(c)$, a contradiction.

$(3)\Rightarrow (4)$:  Let $e\in \m^{feq}$ be a finitary imaginary.  Let $\varphi(x,y)$ be a finitary definable predicate such that $e$ is the canonical parameter for $\varphi(x,a)$.  Let $P(x,c)$ be as in (3) for $\varphi(x,a)$.  We claim that $c$ is the desired tuple.  First suppose that $\sigma\in \Aut(\m^{feq}|c)$.  Then $$\varphi(x,a)\equiv P(x,c)\equiv P(x,\sigma(c))\equiv\varphi(x,\sigma(a)),$$ whence $\sigma(e)=e$.  It follows that $e\in \dcl(c)$.  Now suppose that $\sigma\in \Aut(\m^{feq}|e)$.  Then $P(x,c)\equiv\varphi(x,a)\equiv \varphi(x,\sigma(a))\equiv P(x,\sigma(c))$.  This implies that
$$Y:=\{\sigma(c) \ | \ \sigma\in \Aut(\m^{feq}|e)\}\subseteq X:=\{c' \ | \ c'\equiv c \text{ and } \varphi(x,a)\equiv P(x,c')\}.$$  Since $X$ is compact and $Y$ is closed (it is the set of realizations of $\tp(c/e))$, it follows that $Y$ is compact, i.e. that $c\in \acl(e)$.

$(4)\Rightarrow (1)$:  Let $\varphi(x,a)$ be a finitary definable predicate and let $e\in\m^{feq}$ be a canonical parameter for $\varphi(x,a)$.  Let $c$ be a finite tuple from $\m$ such that $e\in \dcl(c)$ and $c\in \acl(e)$.  We claim that this $c$ is as desired.  Suppose $\sigma \in \Aut(\m|c)$.  Then $\sigma(e)=e$, whence $\varphi(x,a)\equiv \varphi(x,\sigma(a))$.  Thus, $\varphi(x,a)$ is defined over $c$.  Now suppose that $\varphi(x,a)$ is defined over $B$.  Let $\sigma \in\Aut(\m|B)$.  Then $\varphi(x,a)\equiv \varphi(x,\sigma(a))$, i.e. $\sigma(e)=e$.  It follows that $e\in \dcl(B)$, and since $c\in \acl(e)$, we have $c\in \acl(B)$.              
\end{proof}

\begin{df}
Say that $T$ has \textbf{weak elimination of finitary imaginaries} if any of the equivalent conditions of the previous lemma hold.
\end{df}

The following lemma is the continuous analog of a classical lemma due to Lascar.  The classical version can be used to show that the theory of the infinite set has weak elimination of imaginaries.  We will use it in the next section to show that the theory of the Urysohn sphere has weak elimination of finitary imaginaries.

\begin{lemma}\label{L:lascar}
Suppose the following two conditions hold:
\begin{enumerate}
\item There is no strictly decreasing sequence $A_0 \supsetneq A_1 \supsetneq A_2 \supsetneq \ldots$, where each $A_n$ is the real algebraic closure of a finite set of real elements.
\item If $A$ and $B$ are each the real algebraic closure of a finite subset of $\m$ and $\varphi(x,a)$ is a finitary definable predicate which is defined over $A$ and also defined over $B$, then $\varphi(x,a)$ is defined over $A\cap B$.
\end{enumerate}
Then $T$ has weak elimination of finitary imaginaries.
\end{lemma}

\begin{proof}
Let $\varphi(x,a)$ be a finitary definable predicate.  We will verify condition (2) of Lemma \ref{L:WEFI} for $T$.  By (1), there is a finite tuple $c$ such that $\varphi(x,a)$ is defined over $c$ and $\varphi(x,a)$ is not defined over any finite tuple $c'$ such that $\acl(c')\subsetneq \acl(c)$.  Now suppose that $\varphi(x,a)$ is defined over the finite tuple $d$.  We must show that $c\in \acl(d)$.  By (2), $\varphi(x,a)$ is defined over $c\cap d$.  By the choice of $c$, we must have $\acl(c \cap d)=\acl(c)$ which implies that $c\in \acl(d)$.      
\end{proof}

We now aim to show that a real rosy theory which has weak elimination of finitary imaginaries is rosy with respect to finitary imaginaries.  We first need a simplifying lemma.

\begin{lemma}\label{L:reduction}
Suppose that $A$ is a set of finitary imaginaries and $A\nthind_CD$, where $C\subseteq D$.  Then $B\nthind_CD$ where $B$ is a set of real elements for which $\pi(B)=A$.  
\end{lemma}

\begin{proof}
First suppose that $B\ind[M]_CD$.  We show that $A\ind[M]_CD$.  Let $C'\in [C,\acl(D)]$.  Then
$$\acl(BC')\cap \acl(DC')=\acl(C').$$  Since $A\subseteq \dcl(B)$, we have
$$\acl(AC')\cap \acl(DC')\subseteq \acl(BC')\cap \acl(DC')=\acl(C').$$  This shows that $A\ind[M]_CD$.

Now suppose that $B\thind_CD$.  Let $E\supseteq D$.  Then there is $B'\equiv_D B$ such that $B'\ind[M]_CE$.  By the first part of the proof, we have $\pi(B')\ind[M]_CE$.  Since $\pi(B')\equiv_D A$, we see that $A\thind_CD$.   
\end{proof}
\noindent \textbf{Notation:}  Suppose that $T$ has weak elimination of finitary imaginaries.  For a finitary imaginary $e$, we let $l(e)$ denote a real tuple such that $e\in \dcl(l(e))$ and $l(e)\in \acl(e)$.  We refer to $l(e)$ as a \emph{weak code for $e$.}  For a set of finitary imaginaries $E$, we let $l(E):=\bigcup\{l(e) \ | \ e\in E\}$.

\begin{lemma}\label{L:imagreal}
Suppose that $T$ has weak elimination of finitary imaginaries.  Suppose $B\subseteq \m$ and $D\subseteq \m^{feq}$ are small.  Further suppose that $C\subseteq D$ is such that $B\nthind_CD$.  Then $B\nthind_{l(C)}l(D)$ (in the real sense).
\end{lemma}

\begin{proof}
We first show that if $B\ind[M]_{l(C)}l(D)$, then $B\ind[M]_CD$.  Suppose that $C'\in [C,\acl(D)]$.  Then $l(C')\in [l(C),l(\acl(D))]\subseteq [l(C),\acl(l(D))]$.  Since $B\ind[M]_CD$, we have
$$\acl(Bl(C'))\cap \acl(l(D)l(C'))=\acl(l(C')).$$  It follows that
$$\acl(BC')\cap \acl(DC')\subseteq \acl(Bl(C'))\cap \acl(l(D)l(C'))=\acl(l(C'))\subseteq \acl(C').$$  This proves that $B\ind[M]_CD$.

Now suppose that $B\thind_{l(C)}l(D)$.  Suppose $E\supseteq D$.  Then since $l(E)\supseteq l(D)$, there is $B'\equiv_{l(D)}B$ with $B'\ind[M]_{l(C)}l(E)$.  By the first part of the proof, we have $B'\ind[M]_CE$.  Since $D\subseteq \dcl(l(D))$, we have that $B'\equiv_D B$, proving that $B\thind_CD$.      
\end{proof}

\begin{rmk}
The first part of the proof of Lemma \ref{L:imagreal} only used that $T$ had \emph{geometric elimination of finitary imaginaries}, that is, for every $e\in \m^{\feq}$, there is a finite tuple $l(e)$ from $\m$ such that $e$ and $l(e)$ are interalgebraic.  Perhaps a more careful analysis of the second part of the proof could yield that Lemma \ref{L:imagreal} holds under the weaker assumption of geometric elimination of finitary imaginaries.  Also, in the above proof, we never used the fact that each weak code is finite.  In fact, if $\kappa(\m)$ is regular and each weak code is small, then for a small $D\subseteq \m^{\eq}$, $l(D)$ will also be small and the above lemma will hold in this case as well.
\end{rmk}

\begin{thm}\label{T:weakrosy}
Suppose that $T$ has weak elimination of finitary imaginaries and is real rosy.  Then $T$ is rosy with respect to finitary imaginaries.
\end{thm}

\begin{proof}
Let $A\subseteq \m^{feq}$.  We need a cardinal $\kappa(A)$ such that for any $D\subseteq \m^{feq}$, there is $C\subseteq D$ with $|C|\leq \kappa(A)$ and $A\thind_CD$.  Let $B\subseteq \m$ be such that $\pi(B)=A$.  Set $\kappa(A):=\kappa(B)$, where $\kappa(B)$ is understood to be the cardinal that works for $B$ when only considering thorn-forking in the real sense; $\kappa(B)$ exists by real rosiness.  Suppose, towards a contradiction, that $A\nthind_CD$ for all $C\subseteq D$ with $|C|\leq \kappa(A)$.  Then $B\nthind_CD$ for all $C\subseteq D$ with $|C|\leq \kappa(A)$ by Lemma \ref{L:reduction}.  By Lemma \ref{L:imagreal}, we have $B\nthind_{l(C)} l(D)$ for all $C\subseteq D$ with $|C|\leq \kappa(B)$.  Now suppose that $E\subseteq l(D)$ is such that $|E|\leq \kappa(B)$.  Let $C\subseteq D$ be such that $E\subseteq l(C)$ and $|C|\leq \kappa(B)$.  Then $B\nthind_El(D)$ by base monotonicity.  This contradicts the definition of $\kappa(B)$, proving the theorem.  

\end{proof}

\begin{rmk}
Say that $T$ admits \textbf{weak elimination of imaginaries} if, for every $a\in \m^{\eq}$, there is a \emph{countable} tuple $b$ from $\m$ such that $b\in \dcl(a)$ and $a\in \acl(b)$.  The above line of reasoning shows that if $T$ is real rosy and has weak elimination of imaginaries, then $T$ is rosy.
\end{rmk}

\begin{cor}\label{T:weaksuperrosy}
Suppose that $T$ is real superrosy and has weak elimination of finitary imaginaries.  Then $T$ is superrosy with respect to finitary imaginaries.
\end{cor}

\begin{proof}
Let $a\in \m^{feq}$ and $B\subseteq \m^{feq}$ be small.  Let $\epsilon>0$ be given.  Let $a'$ be a tuple from $\m$ be such that $\pi(a')=a$.  Let $\delta>0$ be such that whenever $d(x,y)<\delta$, then $d(\pi(x),\pi(y))<\epsilon$.  Since $T$ is real superrosy, there is a tuple $c'$ from $\m$ such that $d(a',c')<\delta$ and a finite $C\subseteq l(B)$ such that $c'\thind_C l(B)$.  By base monotonicity, we may assume that $C=l(B_0)$ for some finite $B_0\subseteq B$.  By Lemma \ref{L:imagreal}, we have that $c'\thind_{B_0}B$.  Let $c=\pi(c')$.  By Lemma \ref{L:reduction}, we have that $c\thind_{B_0}B$.  By choice of $c'$, we have that $d(a,c)<\epsilon$, completing the proof of the corollary. 
\end{proof}

\section{The Urysohn Sphere}

\

In this section, we present an example of an ``essentially'' continuous theory which is rosy (with respect to finitary imaginaries) but not simple, namely the theory of the Urysohn sphere.  Before proving the main results of this section, let us set up notation and recall some facts about the model theory of the Urysohn sphere.

\

\begin{df}
The \textbf{Urysohn sphere} is the unique (up to isometry) Polish metric space of diameter $\leq 1$ which is \emph{universal} (that is, every Polish metric space of diameter $\leq 1$ can be isometrically embedded into it) and \emph{ultrahomogeneous} (that is, any isometry between finite subsets of it can be extended to an isometry of the whole space).
\end{df}

We let $\U$ denote the Urysohn sphere.  We let $L_\U$ denotes the continuous signature consisting solely of the metric symbol $d$, which is assumed to have diameter bounded by $1$.  We let $T_\U$ denote the $L_\U$-theory of $\U$ and we let $\UU$ denote a monster model for $T_\U$.  We now collect some basic model theoretic facts about the Urysohn sphere, which appear to have been known for a while.  Proofs of these facts can be found in \cite{Usvy}.

\begin{facts}[Henson]

\

\begin{enumerate}
\item $T_\U$ is $\aleph_0$-categorical.
\item $T_\U$ admits quantifier-elimination.
\item $T_\U$ is the model completion of the empty $L$-theory and is the theory of existentially closed metric spaces of diameter bounded by $1$. 
\end{enumerate}
\end{facts}

Another fact about the Ursyohn sphere is that the algebraic closure operator is trivial.  Once again, this fact has been known for a while, but we include here a proof given to us by Ward Henson.

\begin{fact}[Henson]
For every small $A\subseteq \UU$, we have $\acl(A)=\bar{A}$.
\end{fact}

\begin{proof}
The inclusion $\bar{A}\subseteq \acl(A)$ is true in any structure.  Now suppose $b\notin \bar{A}$.  Let $d(b,A)$ denote $\inf\{d(b,a) \ | \ a\in A\}$, a positive number.  Consider the following collection $p(x_i \ | \ i<\omega)$ of closed $L(A)$-conditions:
$$\{d(x_i,a)=d(b,a) \ | \ a\in A, \ i<\omega\} \cup \{d(x_i,x_j)=2\odot d(b,A) \ | \ i<j<\omega\}.$$  It is easy to verify that these conditions define a metric space, whence $p$ can be realized in $\UU$, say by $(b_i \ | \ i<\omega)$.  By quantifier elimination, $\tp(b/A)$ is determined by $\{d(b,a) \ | \ a\in A\}$.  It follows that $b_i\models \tp(b/A)$ for each $i<\omega$.  Since $(b_i \ | \ i<\omega)$ can contain no convergent subsequence, we see that $b\notin \acl(A)$.   
\end{proof}

As the above facts indicate, there appears to be an analogy between the theory of the infinite set in classical logic and the theory of the Urysohn sphere in continuous logic.  However, there is a serious difference between the two theories.  In classical logic, the theory of the infinite set is $\omega$-stable, whereas $T_\U$ is not even simple.  This fact was first observed by Anand Pillay and we provide here a proof communicated to us by Bradd Hart.  

\begin{thm}
$T_\U$ is not simple.
\end{thm}

\begin{proof}
Suppose $A$ is a small set of elements from $\UU$ which are all mutually $\frac{1}{2}$-apart.  Let $p(x)$ be the unique $1$-type over $A$ determined by the conditions $\{d(x,a)=\frac{1}{4} \ | \ a\in A\}$.  It suffices to show that $p$ divides over any proper closed subset $B$ of $A$.  Indeed, suppose $B\subsetneq A$ is closed and $a\in A\setminus B$.  Then, since $a\notin \acl(B)$, we can find $(a_i \ | \ i<\omega)\in \i(a/B)$ such that $d(a_i,a_j)=1$ for all $i<j<\omega$.  Indeed, the set of conditions
$$\Gamma(x_i \ | \ i<\omega):=\{d(x_i,b)=\frac{1}{2}\ | \ i<\omega, \ b\in B\}\cup \{d(x_i,x_j)=1 | \ i<j<\omega\}$$
is finitely satisfiable, and hence realized in $\UU$.  By quantifier elimination in $T_\U$, we have $(a_i \ |  \ i<\omega)\in \i(a/B)$.  We now see that $\{d(x,a_i)=\frac{1}{4} \ | \ i<\omega\}$ is inconsistent, whence the formula $d(x,a)=\frac{1}{4}$ divides over $B$.    
\end{proof}

\begin{rmk}
There are a few more model-theoretic facts about $T_\U$ that are known but have not yet appeared in the literature.  First, since the random graph is a ``subspace'' of $\U$, we see that $T_\U$ is an independent theory (in a rather strong sense).  Berenstein and Usvyatsov have observed that $T_\U$ has SOP$_3$.  Also, Usvyatsov has shown that $T_\U$ does not have the strict order property.
\end{rmk}

We now aim to prove that $T_\U$ is real rosy.  Until further notice, the independence relations $\ind[M] \ $ and $\thind$ will be restricted to the real sorts.  Suppose that $A,B,C$ are small subsets of $\UU$.  Then:

\begin{alignat}{2}
A\ind[M]_C\  B &\Leftrightarrow \text{ for all }C'\in [C,\overline{B\cup C}] (\overline{A\cup C'})\cap (\overline{B\cup C'})=\overline{C'}) \notag \\ \notag
                      &\Leftrightarrow \text{ for all }C'\in [C,\overline{B\cup C}] (\overline{A}\cap \overline{B}\subseteq \overline{C'}) \\ \notag
                      &\Leftrightarrow \overline{A}\cap \overline{B}\subseteq \overline{C}.\notag
\end{alignat}

\

\begin{lemma}
In $T_\U$, $\ind[M]$ satisfies extension, i.e. $\ind[M]=\thind$.
\end{lemma}

\begin{proof}
Since $\ind[M]$ satisfies invariance, monotonicity, transitivity, normality, and symmetry, by Remark 1.2(3) of \cite{Adler}, it suffices to check that $\ind[M]$ satisfies full existence, that is, for any small $A,B,C\subseteq \UU$, we can find $A'\equiv_CA$ such that $A'\ind[M]_CB$.  Let $A,B,C\subseteq \UU$ be small.  Without loss of generality, we may assume that $A,B,C$ are closed.  Indeed, suppose we find $A''\equiv_{\overline{C}}\overline{B}$ with $A''\ind[M]_{\overline{C}}\overline{B}$.  Let $A'\subseteq A''$ correspond to $A$, so $A''=\overline{A'}$.  Then this $A'$ is as desired.

Let $(a_i \ | \ i\in I)$ enumerate $A\setminus C$.  For each $i\in I$, set $$\epsilon_i:=\inf \{d(a_i,c)\ | \ c\in C\}>0.$$  Let $p(X,C):=\tp(A/C)$, where $X=(x_i \ | \ i\in I')$, $I\subseteq I'$, and $(x_i\ | \ i\in I)$ corresponds to the enumeration of $A\setminus C$.  Let $(b_j \ | \ j\in J)$ enumerate $B\setminus C$.  For $i\in I$ and $j\in J$, set $\delta_{i,j}:=\max(d(a_i,b_j),\epsilon_i)$.
Set $$\Sigma:=\Sigma(X):=p(X,C)\cup \{|d(x_i,b_j)-\delta_{i,j}|=0 \ | \ i\in I, j\in J\}.$$

\

\noindent \textbf{Claim:}  $\Sigma$ is satisfiable.

\noindent \textbf{Proof of Claim:}  Let $S:=\{x_i \ | \ i\in I\} \cup B \cup C$.  Let $\rho:S^2\to \r$ be defined as follows:
\begin{itemize}
\item $\rho\upharpoonright (B\cup C)^2=d\upharpoonright (B\cup C)^2$;
\item $\rho(x_{i_1},x_{i_2})=d(a_{i_1},a_{i_2})$ for all $i_1,i_2\in I$;
\item $\rho(x_i,b_j)=\delta_{i,j}$ for all $i\in I$ and all $j\in J$;
\item $\rho(x_i,c)=d(a_i,c)$ for all $i\in I$.
\end{itemize}
It suffices to show that $(S,\rho)$ is a metric space.  Indeed, suppose that $\Sigma_0\subseteq \Sigma$ is finite.  Let $S_0\subseteq S$ be finite such that the parameters and variables occuring in $\Sigma_0$ are from $S_0$.  Since $(S_0,\rho)$ is a finite metric space of diameter bounded by $1$, it is isometrically embeddable in $\U$.  By the strong homogeneity of $\UU$, we may assume that the embedding $\iota:S_0\to \UU$ is such that $\iota(y)=y$ for all $y\in S_0\cap (B\cup C)$.  Let $X_0\subset X$ be the variables appearing in $S_0$.  Since $T_\U$ admits quantifier elimination, it follows that $\iota(X_0)$ realizes $\Sigma_0$.  By compactness, $\Sigma$ is satisfiable.

In order to check that $(S,\rho)$ is a metric space, we must show that, for any $s_1,s_2,s_3\in S$, we have $\rho(s_1,s_2)\leq \rho(s_1,s_3)+\rho(s_2,s_3)$.  We distinguish this into cases, depending on what part of $S$ the $s_i$'s come from.  For example, Case ABC is the case when $s_1\in X$, $s_2\in B\setminus C$, and $s_3\in C$.  There are 15 cases for which there is either no $A$ or no $B$; these cases are trivially true.  Let us turn our attention to the remaining 12 cases.

Consider Case ACB=Case CAB.  We must show that $\rho(x_i,c)\leq \rho(x_i,b)+\rho(b,c)$.  However, we have $\rho(x_i,c)=d(a_i,c)\leq d(a_i,b)+d(b,c)\leq \rho(x_i,b)+\rho(b,c)$.  It is easily verified that this same argument handles Cases BCA, CBA, BBA, and AAB.

Next consider Case ABC=Case BAC.  We need to show that $\rho(x_i,b)\leq \rho(x_i,c)+\rho(c,b)=d(a_i,c)+d(c,b)$.  If $\rho(x_i,b)=d(a_i,b)$, then the result is clear.  Otherwise $\rho(x_i,b)=\epsilon_i\leq d(a_i,c)$, and the result is once again clear.

Next consider Case ABB=Case BAB.  We need to show that $\rho(x_i,b)\leq \rho(x_i,b')+\rho(b,b')$.  If $\rho(x_i,b)=d(a_i,b)$, then we have $\rho(x_i,b)=d(a_i,b)\leq d(a_i,b')+d(b,b')\leq \rho(x_i,b')+\rho(b,b')$.  Otherwise, $\rho(x_i,b)=\epsilon_i\leq \rho(x_i,b')$, and the inequality once again holds.

Finally, consider Case ABA=Case BAA.  We need to show that $\rho(x_{i_1},b)\leq \rho(x_{i_1},x_{i_2})+\rho(x_{i_2},b)$.  Set $r:=\rho(x_{i_1},x_{i_2})=d(a_{i_1},a_{i_2})$.  First suppose that $\rho(x_{i_1},b)=d(a_{i_1},b)$.  Then $\rho(x_{i_1},b)=d(a_{i_1},b)\leq r+d(a_{i_2},b)\leq r+\rho(x_{i_2},b)$.  Now suppose that $\rho(x_{i_1},b)=\epsilon_i$.  To handle this case, we need to first observe that $\epsilon_{i_1}\leq r+\epsilon_{i_2}$.  Indeed, let $c\in C$ be arbitrary.  Then $d(a_{i_1},c)\leq r+d(a_{i_2},c)$.  It follows that $d(a_{i_1},c)\leq d+\epsilon_{i_2}$.  Since $\epsilon_{i_1}\leq d(a_{i_1},c)$, we have that $\epsilon_{i_1}\leq r+\epsilon_{i_2}$.  But now $\rho(x_{i_1},b)=\epsilon_{i_1}\leq r+\epsilon_{i_2}\leq r+\rho(x_{i_2},b)$.  This finishes the proof of this case as well as the proof of the claim.

\

By the Claim, we can find $A'\models \Sigma(X)$.  We claim that this $A'$ is as desired.  Indeed, suppose that $e\in (A'\cap B)\setminus C$.  Let $i\in I$ be such that $e$ corresponds to $x_i$.  Then $0=d(e,e)=d(x_i,e)\geq \epsilon_i>0$, a contradiction.           
\end{proof}

\

\begin{thm}\label{T:urysohnrealrosy}
$T_\U$ is real rosy.
\end{thm}

\begin{proof}
We must show that $\thind$ satisfies local character.  By the previous lemma, this amounts to showing that $\ind[M]\ $ satisfies local character.  Let $A$ and $B$ be small subsets of $\UU$.  For each $x\in \overline{A}\cap \overline{B}$, let $B_x\subseteq B$ be countable such that $x\in \overline{B_x}$.  Let $C:=\bigcup \{B_x \ | \ x\in \overline{A}\cap \overline{B}\}$.  Then $\overline{A}\cap \overline{B}\subseteq \overline{C}$, i.e. $A\ind[M]_CB$.  Since $|C|\leq |\overline{A}|\cdot \aleph_0$, $\ind[M]\ $ has local character.
\end{proof}

\begin{cor}\label{T:urysohnrealsuperrosy}
$T_\U$ is real superrosy.
\end{cor}

\begin{proof}
Let $a\in \UU^n$, $B\subseteq \UU$ small, and $\epsilon>0$.  Write $a=(a_1,\ldots,a_n)$.  Fix $i\in \{1,\ldots,n\}$.  If $a_i\in \acl(B)=\overline{B}$, set $c_i$ to be an element of $B$ such that $d(a_i,c_i)<\epsilon$.  If $a_i\notin \acl(B)$, set $c_i:=a_i$.  Let $c=(c_1,\ldots,c_n)\in \m^n$.  Let $B_0=\{c_1,\ldots,c_n\}\cap B$.  Then $c\thind_{B_0}B$, finishing the proof of the corollary.  
\end{proof}

\

In order to prove that $T_\U$ has weak elimination of finitary imaginaries, we will need the following fact due to Julien Melleray.

\begin{fact}[\cite{Melleray}]\label{L:julien}
Let $A$ and $B$ be finite subsets of $\UU$.  Let $G:=\operatorname{Aut}(\UU | A\cap B)$ and $H:=$ the subgroup of $G$ generated by $\operatorname{Aut}(\UU|A) \cup \operatorname{Aut}(\UU|B)$.  Then $H$ is dense in $G$ with respect to the topology of pointwise convergence.
\end{fact}

\begin{lemma}\label{L:urysohnWEFI}
$T_\U$ has weak elimination of finitary imaginaries.
\end{lemma}

\begin{proof}
We verify properties (1) and (2) of Lemma \ref{L:lascar} for $T_\U$.  Since real algebraic closures of finite subsets of $\UU$ are finite, property (1) is clear.  We now verify (2).  Let $A$ and $B$ be finite subsets of $\UU$.  Let $\varphi(x)$ be a finitary definable predicate which is defined over $A$ and defined over $B$.  We must show that $\varphi(x)$ is defined over $A\cap B$.  Once again, let $G=\operatorname{Aut}(\UU | A\cap B)$ and $H=$ the subgroup of $G$ generated by $\operatorname{Aut}(\UU|A)\cup \operatorname{Aut}(\UU|B)$.  Fix $a\in \UU_x$.  Note that if $\tau\in H$, then $\varphi(\tau(a))=\varphi(a)$.  Now suppose that $\tau \in G$.  By Fact \ref{L:julien}, there is a sequence $(\tau_i \ | \ i<\omega)$ from $H$ such that $\tau_i(a)\to \tau(a)$.  Since $\varphi$ is continuous, we have $\varphi(\tau(a))=\lim \varphi(\tau_i(a))=\varphi(a)$.  Since $a$ was arbitrary, this shows that $\varphi(\tau(x))\equiv \varphi(x)$.  Since $\tau\in G$ was arbitrary, we have that $\varphi(x)$ is defined over $A\cap B$, completing the proof of the lemma.   
\end{proof}

\begin{cor}
$T_\U$ is rosy with respect to finitary imaginaries.
\end{cor}

\begin{proof}
This is immediate from Theorem \ref{T:weakrosy}, Theorem \ref{T:urysohnrealrosy}, and Lemma \ref{L:urysohnWEFI}.
\end{proof}

\begin{cor}
$T_{\U}$ is superrosy with respect to finitary imaginaries.
\end{cor}

\begin{proof}
This is immediate from Corollary \ref{T:weaksuperrosy}, Corollary \ref{T:urysohnrealsuperrosy}, and Lemma \ref{L:urysohnWEFI}.
\end{proof}

We end this section with an application of the fact that $T_\U$ is real rosy.  For $p\in S(A)$, one defines $U^\th(p)$ as in classical model theory.  If $X$ is an $A$-definable set, one defines $U^\th(X):=\sup \{U^\th(a/A)\ | \ a\in X\}$.  
If $U^\th(X)<\omega$, then there is $a\in X$ such that $U^\th(X)=U^\th(a/A)$.  The Lascar inequalities for $U^\th$-rank also hold in this context.

\begin{prop}\label{L:Urank}
Suppose $f:\m^n\to \m$ is an injective $A$-definable function, where $A\subseteq \m^{\eq}$ is small.  Suppose that $U^\th(\m)<\omega$.  Then $U^\th(\m^n)\leq U^\th(f(\m^n))$.
\end{prop}    

\begin{proof}
Let $a\in \m^n$ be such that $U^\th(\m^n)=U^\th(a/A)$.  Let $b:=f(a)$.  Then since $a$ and $b$ are interdefinable over $A$, we have, by the Lascar inequalities, that $U^\th(b/A)=U^\th(ab/A)=U^\th(a/A)$.  Consequently, we see that $U^\th(\m^n)=U^\th(a/A)=U^\th(b/A)\leq U^\th(f(\m^n))$.    
\end{proof}

Define $U^\th_{\real}$ and $U^\th_{\feq}$ to be the foundation rank of $\thind$ when restricted to the real sorts and finitary imaginary sorts respectively.  The previous proposition continues to hold when $U^\th$ is replaced by $U^\th_{\real}$ or $U^\th_{\feq}$.

\begin{lemma}
For each $n>0$, we have $U^\th_{\real}(\UU^n)=n$.
\end{lemma}

\begin{proof}
We prove this by induction on $n$.  First suppose that $n=1$.  Let $p$ be the unique element of $S_1(\emptyset)$.  Since $p$ is consistent, we have $U^\th(p)\geq 0$.  Let $a\models p$.  Since $\tp(a/a)$ $\th$-forks over $\emptyset$, we see that $U^\th_{\real}(p)\geq 1$.  Suppose $U^\th_{\real}(p)\geq 2$.  Then there would be $b$ and $A$ such that $U^\th_{\real}(b/A)\geq 1$ and $\tp(b/A)$ $\th$-forks over $\emptyset$, i.e. $b\in \bar{A}$.  Since $b\in \acl(A)$, $\tp(b/A)$ cannot have a $\th$-forking extension, contradicting $U^\th_{\real}(b/A)\geq 1$.  Thus $U^\th_{\real}(p)=1$ for the unique type in $S_1(\emptyset)$, whence $U^\th(\UU)=1$.  Now suppose that $n>1$.  Let $a\in \UU^{n-1}$ be such that $U^\th_{\real}(\UU^{n-1})=U^\th_{\real}(a/\emptyset)$.  Let $b\in \UU$ be such that $b$ does not equal any of the coordinates of $a$.  Then $a\thind b$, so by the Lascar inequalities, $U^\th_{\real}(ab)=U^\th_{\real}(a)+U^\th_{\real}(b)=(n-1)+1=n$.  It follows that $U^\th_{\real}(\UU^n)\geq n$.  However, for any $c\in \UU^{n-1}$ and $d\in \UU$, we have $U^\th_{\real}(cd)\leq U^\th_{\real}(c/d)\oplus U^\th_{\real}(d)\leq n$, whence $U^\th_{\real}(\UU^n)\leq n$.  Thus, $U^\th_{\real}(\UU^n)=n$.
\end{proof}

\begin{cor}
For each $n>0$, we have $U^\th_{\feq}(\UU^n)=n$.
\end{cor}

\begin{proof}
To prove the corollary, it suffices to show that, for any $a\in \UU^n$, we have $U^\th_{\feq}(a/\emptyset)\leq U^\th_{\real}(a/\emptyset)$.  However, this follows immediately from Lemma \ref{L:imagreal}.
\end{proof}

By the universality property of the Urysohn sphere, we have that, for $n>1$, $\U^n$ isometrically embeds into $\U$.  The next corollary shows that this cannot be done \emph{definably}.

\begin{cor}

\

\begin{enumerate}
\item For any $n\geq 2$, there does not exist a definable isometric embedding $f:\U^n\to \U$.
\item For any $n\geq 2$, there does not exist an $A$-definable injective function $f:\U^n\to \U$, where $A\subseteq \U$ is finite..
\end{enumerate}
\end{cor}

\begin{proof}
In either case, if such an $f$ existed, then by Lemma \ref{L:injectiveextension}, the natural extension $g:\UU^n\to \UU$ of $f$ to $\UU^n$ would be injective.  By Proposition \ref{L:Urank}, we would have $$n=U^\th_{\real}(\UU^n)\leq U^\th_{\real}(g(\UU^n))\leq U^\th_{\real}(\UU)=1,$$ a contradiction.
\end{proof}

Ward Henson has a more elementary proof that, assuming $\kappa(\UU)>2^{\aleph_0}$, there can be no definable, injective function $f:\UU^n\to \UU$ for any $n\geq 2$.  It suffices to treat the
case $n=2$, as if $n>2$, we specify the extra coordinates arbitrarily in $\UU$,
getting a definable injective function $\UU^2 \to \UU$.  Let $A$ be a closed separable set on which $f$ is definable.  For any $a \in \UU^2$, we have $f(a)\in \dcl(Aa)=Aa$.  So, if $f(a)\notin A$, then $f(a)$ equals one of the coordinates of $a$.  Since $f$ is injective, $|f^{-1}(A)|\leq 2^{\aleph_0}$.  Let $S$ be a continuum sized subset of $\UU$ such that $f^{-1}(A)\subseteq S^2$.  Then on $(\UU\setminus S)^2$, the function $f$ is always equal to one of its coordinates.  Let $F\subseteq \UU\setminus S$ have cardinality $2$.  Then $|f(F^2)|\leq 2$.  However, since $f$ is injective, $|f(F^2)|=4$.  This contradiction proves that such an $f$ could not exist.

\section{Other Notions of Thorn-forking}

In this section, we discuss other natural ways of defining thorn-forking in continuous logic and show that they also yield well-behaved independence relations.  Throughout this section, we work in $\m^{\eq}$.

\begin{df}
Let $\varphi(x,y)$ be a formula.
\begin{enumerate}
\item We say that $\varphi(x,c)$ \textbf{maximally strongly divides over $B$} if it strongly $\chi(c/B)$-divides over $B$.
\item We say that $\varphi(x,c)$ \textbf{maximally $\th$-divides over $B$} if there is $D\supseteq B$ so that $\varphi(x,c)$ maximally strongly divides over $D$.
\item We say that the partial type $\pi(x)$ (in possibly infinitely many variables) \textbf{maximally $\th$-forks over $B$} if there is a cardinal $\lambda<\kappa(\m)$ and formulae $\varphi_i(x,c_i)$, $i<\lambda<\kappa(\m)$, such that each $\varphi_i(x,c_i)$ maximally $\th$-divides over $B$ and such that $$\z(\pi(x))\subseteq \bigcup_{i<\lambda} \z(\varphi_i(x,c_i)).$$
\item We say that $A\ind[m\th]_C B$ if $\tp(A/BC)$ does not maximally $\th$-fork over $C$.
\item We say that $\varphi(x,c)$ \textbf{maximally strongly divides over $B$ in the na\"ive sense} if it strongly $\chi(c/B)$-divides over $B$ in the na\"ive sense.  One can then define \textbf{maximally $\th$-dividing in the na\"ive sense} and \textbf{maximally $\th$-forking in the na\"ive sense} in the obvious way.
\end{enumerate}
\end{df}

\begin{lemma}\label{L:na\"ive}
For every $A,B,C$, we have $A\ind[m\th]_C B$ if and only $A\ind[m\th]_C B$ in the na\"ive sense.
\end{lemma}

\begin{proof}
The backwards direction being obvious, suppose $A\nind[m\th]_C B$ in the na\"ive sense.  Choose formulae $\varphi_i(x,c^i)$, $i<\lambda$, and parameter sets $D_i$, $i<\lambda$, each containing $C$, such that each $\varphi_i(x,c^i)$ strongly $\chi(c^i/D_i)$-$k_i$-divides over $D_i$ in the na\"ive sense, and such that $$\z(\tp(A/BC))\subseteq \bigcup_{i<\lambda} \z(\varphi_i(x,c^i)).$$  By saturation, for each $i<\lambda$ we can find $\eta_i>0$ such that for any $c_1^i,\ldots,c_{k_i}^i$ realizing $\tp(c^i/D_i)$ which are at least $\chi(c^i/D_i)$-apart, we have $$\inf_x \max_{1\leq j \leq k_i} \varphi_i(x,c_j^i)\geq \eta_i.$$  Let $\psi_i(x,c^i):=\frac{1}{\eta_i}\odot \varphi_i(x,c^i)$.  Then $\z(\psi_i(x,c^i))=\z(\varphi_i(x,c^i))$ and $\psi_i(x,c^i)$ maximally strongly divides over $D_i$, whence we can conclude that $A\nind[m\th]_C B$.   
\end{proof}

\begin{lemma}\label{L:finiteunion}
A partial type $\pi(x)$ maximally $\th$-forks over $B$ if and only if there exists $n>0$ and formulae $\varphi_i(x,c_i)$, $i=1,\ldots,n$, each of which maximally $\th$-divides over $B$, such that $\z(\pi(x))\subseteq \bigcup_{i=1}^n \z(\varphi_i(x,c_i)).$
\end{lemma}

\begin{proof}
Suppose $\z(\pi(x))\subseteq \bigcup_{i<\lambda} \z(\psi_i(x,c_i))$, where each $\psi_i(x,c_i)$ maximally strongly divides over $D_i\supseteq B$.  By compactness, we have $i_1,\ldots,i_n<\lambda$ such that $$\z(\pi(x))\subseteq \bigcup_{j=1}^n \z(\psi_{i_j}(x,c_{i_j})\dotminus \frac{1}{2}).$$  But then the formulae $\varphi_j(x,c_{i_j}):=2\psi_{i_j}\dotminus1$ maximally strongly divide over $D_i$ and $\z(\pi(x))\subseteq \bigcup_{j=1}^n \z(\varphi_j(x,c_{i_j}))$.     
\end{proof}

\begin{lemma}
Suppose $p\in S(C)$.  Then $p$ maximally $\th$-forks over $B$ if and only if there exists an $L(C)$-formula $\varphi(x,c)$ such that the condition $``\varphi(x,c)=0"$ is in $p$ and there exists formulae $\varphi_i(x,c_i)$, $i=1,\ldots, n$, each of which maximally $\th$-divide over $B$, such that $\z(\varphi(x,c))\subseteq \bigcup_{i=1}^n \z(\varphi_i(x,c_i))$.
\end{lemma}

\begin{proof}
This is proven in the exact same way as in the proof of Lemma \ref{L:finiteunion}.
\end{proof}

\begin{lemma}

\

\begin{enumerate}
\item If the formula $\varphi(x,c)$ maximally strongly divides over $B$, then it divides over $B$.
\item Suppose $A \ind_C B$.  Then $A \ind[m\th]_C B$.
\end{enumerate}
\end{lemma}

\begin{proof}
(1)$\Rightarrow$(2) follows from Lemma \ref{L:finiteunion}, so we need only prove (1).  However any $I\in \i(c/B)$ with $d(I)=\chi(c/B)$ witnesses that $\varphi(x,c)$ divides over $B$.  
\end{proof}

\begin{lemma}
Suppose the formula $\varphi(x,c)$ maximally $\th$-divides over $B$, witnessed by maximal strong dividing over $D\supseteq B$.  Then there exists a finite tuple $d\in D$ so that the formula $2\odot \varphi(x,c)$ maximally strongly divides over $Bd$.  Consequently, a partial type $\pi(x)$ maximally $\th$-forks over $B$ if and only if there exists $\varphi_i(x,c_i)$, $i=1,\ldots,n$, and finite tuples $d_1,\ldots,d_n$, so that each $\varphi_i(x,c_i)$ maximally strongly divides over $Bd_i$ and such that $$\z(\pi(x))\subseteq \bigcup_{i=1}^n \z(\varphi_i(x,c_i))).$$
\end{lemma}

\begin{proof}
Let $p(x):=\tp(c/D)$ and $r:=\chi(c/D)$.  Let $k$ be such that $\varphi(x,c)$ maximally strongly $r$-$k$-divides over $D$.  Then the collection of formulae $$p(y_1)\cup \ldots \cup p(y_k) \cup \{d(y_i,y_j)\geq r \ | 1\leq i<j\leq k\} \cup \{\inf_x \max_{1\leq i \leq k} \varphi(x,y_i)\leq \frac{1}{2}\}$$ is inconsistent.  Hence we have a formula $\psi(x,B,d)$, where $d$ is a finite tuple from $D\setminus B$, such that the condition $``\psi(x,B,d)=0"$ is in $\tp(c/D)$ and such that, for all $c_1,\ldots,c_k \in\z(\psi)$ which are pairwise at least $r$-apart, we have $\inf_x \max_{1\leq i \leq k}\varphi(x,c_i)>\frac{1}{2}$.  Since $\chi(c/Bd)\geq \chi(c/D)=r$, it follows that $2\odot \varphi(x,c)$ maximally strongly divides over $Bd$.

\end{proof}

\begin{rmk}
The proof of the above lemma also shows that, like in classical logic, the ``$k$-inconsistency'' in the maximal strong dividing of $\varphi(x,c)$ over $B$ is witnessed by the zeroset of a single formula $\psi(x)$ with parameters from $B$ for which $\psi(c)=0$; see Remark 2.1.2 in \cite{Alf} for the statement of this in the classical setting.
\end{rmk}

\begin{prop}\label{L:thindprop}
Suppose $A,B,C,D$ are small parameter sets.  The following properties of $\ind[m\th]\ $ hold in any theory:
\begin{enumerate}
\item Automorphism Invariance:  For any automorphism $\sigma$, if $A \ind[m\th]_C B$, then $\sigma(A) \ind[m\th]_{\sigma(C)} \sigma(B)$.
\item Extension:  If $B\subseteq C \subseteq D$ and $A\ind[m\th]_B C$, then there is $A'\equiv_CA$ such that $A'\ind[m\th]_B D$.
\item Monotonicity:  If $B\subseteq C \subseteq D$ and $A\ind[m\th]_B D$, then $A\ind[m\th]_C D$.
\item Partial Right Transitivity:  If $B\subseteq C \subseteq D$ and $A\ind[m\th]_B D$, then $A\ind[m\th]_C D$ and $A\ind[m\th]_B C$.
\item Finite Character:  $A\ind[m\th]_C B$ if and only if $a\ind[m\th]_C b$ for all finite tuples $a$ and $b$ from $A$ and $B$ respectively.
\item Base Extension:  If $A\ind[m\th]_C B$, there is $D'\equiv_{BC}D$ such that $A\ind[m\th]_{CD'} B$. 
\item If $C\subseteq B$, we have $A\ind[m\th]_C B$ if and only if $A\ind[m\th]_C \acl(B)$.
\end{enumerate}
\end{prop}

\begin{proof}
(1) is clear.  For (2), suppose $\{p_i \ | \ i<\lambda\}$ enumerate the extensions of $p:=\tp(A/C)$ to $D$.  Suppose, towards a contradiction, that each $p_i$ maximally $\th$-forks over $B$.  Then for each $i$, there are formulae $\varphi_{i,j}(x,c_{i,j})$, $j=1,\ldots,n_i$, such that each $\varphi_{i,j}(x,c_{i,j})$ maximally $\th$-divides over $B$ and $\z(p_i)\subseteq \bigcup_{j=1}^{n_i}\z(\varphi_{i,j}(x,c_{i,j})).$  But then 
$$\z(p)\subseteq \bigcup \{\z(\varphi_{i,j}(x,c_{i,j})) \ | \ i<\lambda, \ j<n_i\},$$ whence $p$ maximally $\th$-forks over $C$, a contradiction.

(3)  This follows from the fact that if $\varphi(x,c)$ maximally $\th$-divides over $C$, then it maximally $\th$-divides over $B$.

(4)  The first claim is just monotonicity and the second claim follows from the fact that $\tp(A/C)\subseteq \tp(A/D)$.

(5)  First suppose that $A\nind[m\th]_C B$.  Then we have a formula $\varphi(x,b,c)$ which maximally $\th$-forks over $C$ and such that the condition $\varphi(x,b,c)=0$ is in $\tp(A/BC)$.  Let $a$ be a tuple from $A$ such that $\varphi(a,b,c)=0$.  Then $\varphi(x,b,c)$ witnesses that $a\nind[m\th]_C b$.  Now suppose  $A\ind[m\th]_C B$ and $a$ and $b$ are finite tuples from $A$ and $B$ respectively.  Then since $\tp(a/bC)\subseteq \tp(A/BC)$, we have $a\ind[m\th]_C b$.  

(6)  By extension, we can find $A'\models \tp(A/BC)$ with $A'\ind[m\th]_C BD$.  But then by monotonicity, we have $A'\ind[m\th]_{CD} B$.  By automorphism invariance, we have $D'\equiv_{BC}D$ such that $A\ind[m\th]_{CD'}B$.

(7)  One direction is clear by monotonicity.  Now let $a$ be a finite tuple from $A$ and suppose $a\nind[m\th]_C \acl(B)$.  Choose an $L(\acl(B))$-formula $\varphi(x,d)$ which maximally $\th$-forks over $C$ and such that $\varphi(a,d)=0$.  By Lemma 1.8 in \cite{HenTell}, $d\in \bdd(B)$, whence we may enumerate $\z(\tp(d/B))=\{d_i \ | \ i<\lambda\}$, where $\lambda<\kappa(\m)$.  Note that each $\varphi(x,d_i)$ maximally $\th$-forks over $C$.  From this and the fact that $$\z(\tp(a/B))\subseteq \bigcup_{i<\lambda} \z(\varphi(x,d_i)),$$ we see that $a\nind[m\th]_C B$.
\end{proof}

The following lemma is the analog of Proposition \ref{L:aclstrongdivide} for maximal strong dividing.

\begin{lemma}
Suppose $\varphi(x,c)$ maximally strongly divides over $B$ and $\varphi(a,c)=0$.  Then $\chi(c/Ba)<\chi(c/B)$.
\end{lemma}

\begin{proof}
Suppose $I\in \i(c/Ba)$.  Then since $\varphi(a,c')=0$ for each $c'\in I$, we must have $d(I)<\chi(c/B)$, else we contradict strong dividing. 
\end{proof}

Using the preceding lemma, we prove the next theorem exactly like we proved Theorem \ref{T:weakest}.

\begin{thm}
Suppose $\ind[I]$ is an automorphism ternary relation on small subsets of $\m$ satisfying:
\begin{enumerate}
\item for all finite tuples $b$, if $\chi(b/AC)< \chi(b/C)$, then $A\nind[I]_Cb$;
\item for all $A,B,C,D$, if $A\ind[I]_B D$ and $B\subseteq C \subseteq D$, then $A\ind[I]_C D$ and $A\ind[I]_BC$;
\item for all $A,B,C,D$, if $A\ind[I]_C B$ and $D\supseteq BC$, then there is $A'\equiv_{BC}A $ such that $A'\ind[I]_CD$.
\item for all $A,B,C$, if $A\ind[I]_CBC$, then $A\ind[I]_CB$.
\end{enumerate}
Then for all $A,B,C$, $A\ind[I]_CB \Rightarrow A\ind[m\th]_CB$.
\end{thm}

The proof of the following lemma has a similar proof to the proof of Lemma 2.1.8 in \cite{Alf}.

\begin{lemma}\label{L:indchi}
Suppose $a\ind[m\th]_A b$.  Then $\chi(b/Aa)=\chi(b/A)$.
\end{lemma}

\begin{proof}
The result is obvious if $\chi(b/A)=0$, so let us assume that $\chi(b/A)>0$.  It suffices to construct $I\in \i(b/Aa)$ with $d(I)=\chi(b/A)$.  Let $p(x,y):=\tp(a,b/A)$.  Note that, by Lemma \ref{L:na\"ive}, there is no $L(Ab)$-formula $\varphi(x,b)$ such that $\varphi(a,b)=0$ and $\varphi(x,b)$ maximally strongly divides over $A$ in the na\"ive sense.  Hence, for every such formula $\varphi(x,b)$ and $k<\omega$, there are $b_1,\ldots,b_k$ realizing $\tp(b/A)$ which are at least $\chi(b/A)$-apart and for which there exists $c$ such that $\varphi(c,b_i)=0$ for all $i=1,\ldots,k$.  It thus follows by compactness that the set of conditions $$\bigcup_{i<\omega} p(x,y_i)  \cup \{d(y_i,y_j)\geq \chi(b/A) \ | \ i<j<\omega\}$$ is consistent, say realized by $a_1,J_1$.  By Ramsey's theorem and compactness, we can find an $Aa_1$-indiscernible sequence $J_2$ with $a_1b'$ realizing $p(x,y)$ for each $b'\in J_2$ and such that $d(J_2)=\chi(b/A)$.  Fix $b'\in J_2$.  Let $\sigma\in \Aut(\m/A)$ be such that $\sigma(a_1)=a$ and $\sigma(b')=b$.  Then $I:=\sigma(J_2)$ is as desired.    
\end{proof}

The proof of the following lemma is essentially the same as in the classical case; see \cite{Alf} Lemma 2.1.6.

\begin{lemma}
In any continuous theory $T$, $\ind[m\th]\ $ satisfies Partial Left Transitivity:  For any tuples $a,b,c$ and any parameter set $A$, if $a\ind[\th]_A c$ and $b\ind[m\th]_{Aa} c$, then $ab\ind[m\th]_A c$. 
\end{lemma}

\begin{proof}
Suppose that $a\ind[\th]_A c$ and $b\ind[m\th]_{Aa} c$.  As in the proof in the classical case, it is enough to show that there is no $L(Ac)$-formula $\varphi(x,y,c)$ such that $\varphi(a,b,c)=0$ and $\varphi(x,y,c)$ maximally $\th$-\emph{divides} over $A$ (This reduction in the classical case only uses Extension and automorphisms).  Suppose, towards a contradiction, that there is an $L(Ac)$-formula $\varphi(x,y,c)$ with $\varphi(a,b,c)=0$ and $\varphi(x,y,c)$ maximally $\th$-divides over $A$, say maximally strongly divides over $Ad$.  By base extension, we can find $d'\models \tp(d/Ac)$ for which $a\ind[m\th]_{Ad'}c$.  Since $\varphi(x,y,c)$ still maximally strongly divides over $Ad'$, we may assume $d=d'$, i.e. that $a\ind[m\th]_{Ad} c$.  By Lemma \ref{L:indchi}, we know that $\chi(c/Ada)=\chi(c/Ad)$.  Hence, we have that $\varphi(a,y,c)$ maximally strongly divides over $Ada$, and hence maximally $\th$-divides over $Aa$.  This contradicts the fact that $b\ind[m\th]_{Aa} c$.    
\end{proof}

\begin{df}
We say that $T$ is \textbf{maximally rosy} if $\ind[m\th]\ $ satisfies local character.
\end{df}

\begin{lemma}
In a maximally rosy theory, $\ind[m\th] \ \ $ satisfies Existence:  for all $A,B$, we have $A\ind[m\th]_B B$.
\end{lemma}

\begin{proof}
If not, then the constant sequence $(\tp(A/B))$ contradicts local character.
\end{proof}

From existence, one can quite easily get that, in maximally rosy theories, $\ind[m\th]\ $ is an independence relation.  In particular, by Theorem 2.5 in \cite{Adler}, $\ind[m\th] \ $ satisfies symmetry in maximally rosy theories.

\begin{lemma}
In a maximally rosy theory, $\ind[m\th]\ $ satisfies Anti-reflexivity:  for all $A,B$, we have $A\ind[m\th]_B A$ if and only if $A\subseteq \acl(B)$.
\end{lemma}

\begin{proof}
First suppose that $A\nsubseteq \acl(B)$, i.e. that $\chi(a/B)>0$ for some $a\in A$.  Since the formula $d(x,a)$ maximally strongly divides over $B$ in the na\"ive sense, we see that $a\nind[m\th]_B a$ in the na\"ive sense.  Hence, by Lemma \ref{L:na\"ive}, we see that $a\nind[m\th]_B a$.  By finite character, we conclude that $A\nind[m\th]_B A$.  (Note that this direction did not use the maximal rosiness assumption.)

Now suppose $A\subseteq \acl(B)$.  By existence, we have $A\ind[m\th]_B B$.  By Lemma \ref{L:thindprop} (7), we see that $A\ind[m\th]_B \acl(B)$.  By monotonicity, we conclude that $A\ind[m\th]_B A$.
\end{proof}

\begin{rmk}
In maximally rosy theories, we have that $\ind[m\th] \ \ $ is a strict independence relation.  The fact that $\ind[m\th] \ $ satisfies finite character might make some want to favor it over $\thind$.  However, being maximally rosy seems like quite a strong condition on a theory.  For example, one can show that a classical rosy theory $T$ need not be maximally rosy when viewed as a continuous theory.
\end{rmk}

\begin{df}

\

\begin{enumerate}
\item Say that $\varphi(x,c)$ \textbf{$\th$-$\epsilon$-divides over $A$} if there is $B\supseteq A$ such that $\varphi(x,c)$ strongly-$\epsilon$-divides over $B$.  Say that $\pi(x)$ \textbf{$\th$-$\epsilon$-forks over $A$} if there exists $\lambda<\kappa(\m)$ and formulae $\varphi_i(x,c^i)$, $i<\lambda$, each of which $\th$-$\epsilon$-divide over $A$, and such that $\z(\pi(x))\subseteq \bigcup_{i=1}^n \z(\varphi_i(x))$.  Let $\ethind\ $     
denote the corresponding independence relation.  Say that $T$ is \textbf{$\epsilon$-rosy} if $\ethind \ $ satisfies local character.
\item Say $A\ind[s\th]_C B$, read $A$ is \textbf{strongly thorn-independent from} $B$ over $C$, if there exists $\epsilon>0$ such that $A\ethind_C B$.  Say that $T$ is \textbf{strongly rosy} if $\ind[s\th]\ $ satisfies local character.
\end{enumerate}
\end{df}

\begin{lemma}
Suppose $a\ethind_A b$.  If $\chi(b/A)\geq \epsilon$, then $\chi(b/Aa)\geq \epsilon$.  
\end{lemma}

\begin{proof}
Exactly as in the proof of Lemma \ref{L:indchi}.
\end{proof}

\begin{lemma}
$\ethind\ $ satisfies Partial Left Transitivity.
\end{lemma}

\begin{proof}
Follows from the previous lemma in the exact same way that Partial Left Transitivity for $\ind[m\th]\ \ $ followed from Lemma \ref{L:indchi}.
\end{proof}

It is straightforward to check that $\ethind\ \ $ satisfies all of the other properties of a strict independence relation in an $\epsilon$-rosy theory.  In a strongly rosy theory, $\ind[s\th]\ $ satisfies all of the axioms of a strict countable independence relation.  To verify countable character, suppose that $A,B,C$ are small parametersets.  Suppose that $A_0\subseteq A$ and $B_0\subseteq B$ are countable.  Then $$A\ind[s\th]_C \ B \Rightarrow A\ethind_C \ B (\text{ some }\epsilon >0) \Rightarrow A_0\ethind_C \ B_0 \Rightarrow A_0\ind[s\th]_C \ B_0.$$  Next suppose that $A\nind[s\th]_C \ B$.  Then for every $n>0$, we have $A\nind[\th,\frac{1}{n}]_C \ B$.  Thus, for every $n>0$, we have $\varphi_n(x,b_n)\in \tp(A/BC)$ which $\th$-$\frac{1}{n}$-forks over $C$.  Let $a_n\in A$ be such that $\varphi_n(a_n,b_n)=0$.  Let $A_0=\bigcup_{n>0} a_n$ and let $B_0:=\bigcup_{n>0} b_n$.  Then $A_0\nind[s\th]_C \ B_0$.  Indeed, given $\epsilon>0$, choose $n$ such that $\frac{1}{n}<\epsilon$.  Then $\varphi_n(x,b_n)$ $\th$-$\epsilon$-forks over $C$ and $\varphi_n(x,b_n)\in \tp(A_0/B_0C)$.

\begin{lemma}
For any $\epsilon>0$, we have $$\ind\Rightarrow \ind[\th,\epsilon] \ \Rightarrow \ind[s\th]\ \Rightarrow\thind$$ and $$\ind \Rightarrow \ind[m\th] \ \Rightarrow \thind.$$  Consequently we have $$simple \Rightarrow \epsilon-rosy \Rightarrow strongly\ rosy \Rightarrow rosy$$ and $$simple \Rightarrow maximally\ rosy \Rightarrow rosy.$$
\end{lemma}

\begin{proof}
It is clear that strong $\epsilon$-dividing implies dividing.  This takes care of each of the first implications. The second implication of the first line is true by definition.  The remaining two implications follow from the fact that $\thind$ is weakest amongst the strict countable independence relations.
\end{proof}

\noindent Note that if $\epsilon<\epsilon'$, then strong $\epsilon$-dividing implies strong $\epsilon'$-dividing, so $\epsilon'$-rosy implies $\epsilon$-rosy.  We thus make the following definition.

\begin{df}
$\th(T):=\sup\{\epsilon \ | \ T \text{ is }\epsilon\text{-rosy}\}$.
\end{df}

\begin{question}
Note that if $\th(T)>0$, then $T$ is strongly rosy.  Is the converse true?  What can we say about theories for which $\th(T)=1$?  
\end{question}

\begin{question}
It appears that the argument showing that $T_\U$ is not simple also shows that $T_\U$ is not maximally rosy.  Are there natural examples of maximally rosy theories or strongly rosy theories?
\end{question}

\section{Keisler Randomizations and Rosiness}

In \cite{Keisler}, Keisler introduced the notion of the \emph{randomization} of a theory $T$, denoted $T^R$.  The models of $T^R$ are essentially spaces of $M$-valued random variables, where $M\models T$.  In \cite{BK}, the randomization of a classical theory was phrased in the framework of continuous logic and its properties were further studied.  In \cite{Keisler}, \cite{BK}, and \cite{B3},  theorems of the form ``$T$ is $P$ if and only if $T^R$ is $P$'' were proven, where $P$ stands for any of the following properties:  $\omega$-categorical, $\omega$-stable, stable, NIP.  However, in \cite{B2}, it is shown that if $T$ is simple, unstable, then $T^R$ is not simple.  It is a natural question to ask whether $T$ is rosy if and only if $T^R$ is rosy with respect to finitary imaginaries.  Since the direction ``$T^R$ is $P$ implies $T$ is $P$'' is generally trivial, we tried to prove that if $T^R$ is rosy with respect to finitary imaginaries, then $T$ is rosy.  We were unable to prove that and instead were only able to prove that $T$ is rosy provided $T^R$ is maximally rosy with respect to finitary imaginaries.  We devote this section to proving this fact.

In this section, we assume that the reader is familiar with the basic properties of the Keisler randomization process.  We refer the reader to \cite{BK} for information about the randomization theory.  We also borrow notation from the aforementioned paper.  The set-up for this section differs from earlier parts of this paper.  Let $L$ be a countable classical signature and let $T$ be a complete $L$-theory.  Let $M\models T$ be a monster model.  Let $\kappa>|M|^{2^{\aleph_0}}$ be a cardinal and let $\bo{\m}$ be a monster model of $T^R$ (in the \emph{1-sorted language} $L^R$) which is $\kappa$-saturated and strongly $\kappa$-homogeneous.  By Corollary 2.8 of \cite{BK}, we may assume that $\bo{\m}$ is the structure associated to some full randomization $\K$ of $M$ based on the atomless finitely-additive measure space $(\Omega,\B,\mu)$.  We may further assume that $(\Omega,\B,\mu)$ is $\sigma$-additive; see Theorem 3.6 of \cite{B2}, noting that in our situation, the so-called auxiliary sort is the same as our boolean algebra sort.  Let $\bo{\m}_C$ be the substructure of $\bo{\m}$ which is the structure associated to the elements of $\K$ with countable range.  From now on, for any $a\in M$, we write $\bo{a}$ for the element of $\bo{\m}_C$ which is the equivalence class of the element of $\K$ with constant value $a$.  We do the same for tuples and parameter sets from $M$.  

\

\begin{lemma}\label{L:dettype}
Let $c$ be a finite tuple from $M$ and let $B\subseteq M$ be countable.  Suppose $C\in \K$ is such that $C\models \tp(\bo{c}/\bo{B})$.  Then $C(\omega)\models \tp(c/B)$ for almost all $\omega \in \Omega$.
\end{lemma}

\begin{proof}
Let $\psi(x,b)\in \tp(c/B)$.  Then the condition $\p\l \psi(X,\textbf{b})\rr=1$ is in $\tp(\textbf{c}/\textbf{B})$, whence $\p\l \psi(C,\textbf{b})\rr=1$.  Since $\tp(c/B)$ is countable and $(\Omega,\B,\mu)$ is $\sigma$-additive, we achieve the desired result.  
\end{proof}

\

\begin{lemma}
Suppose $c$ is a finite tuple from $M$ and $B\subseteq M$ is a small parameterset such that $c\notin \acl(B)$.  Then $\chi(\bo{c}/\bo{B})=1$.
\end{lemma}

\begin{proof}
Let $(c_i : i\in \omega)$ be a nonconstant $B$-indiscernible sequence of realizations of $\tp(c/B)$.  Then setting $I:=(\bo{c_i} : i\in \omega)$, we see that $I\in \i(\bo{c}/\bo{B})$ with $d(I)=1$.  Indeed, since $T^R$ admits (strong) quantifier elimination (see \cite{BK}, Theorem 2.9), $\tp(\bo{c_{i_1}},\ldots,\bo{c_{i_n}}/\bo{B})$ is determined by the values of $\p\l \psi(\bo{c_{i_1}},\ldots,\bo{c_{i_n}})\rr$ as $\psi$ ranges over all $L$-formulae with $n$ free variables.  But $\l \psi(\bo{c_{i_1}},\ldots,\bo{c_{i_n}})\rr=\l \psi(\bo{c_{j_1}},\ldots,\bo{c_{j_n}})\rr$ whenever $i_1<\cdots<i_n<\omega$ and $j_1<\cdots<j_n<\omega$ by indiscernibility of $(c_i \ : i\in \omega)$.
\end{proof}




In order to prove the main lemma relating strong dividing in $T$ and maximal strong dividing in $T^R$, we first need to prove a Ramsey-theoretic fact for Boolean algebras equipped with a finitely-additive measure (Lemma \ref{L:ramsey} below).  We had a rather lengthly (nonstandard) proof of the desired fact, but we are grateful to Konstantin Slutsky for showing us the much simpler proof that appears below.

\begin{lemma}
Suppose $B$ is a boolean algebra and $\mu:B\to [0,1]$ is a finitely-additive measure.  Then for any $m>0$ and any set of distinct elements $\{a_1,\ldots,a_{2m}\}$ from $B$ with $\mu(a_i)\geq \frac{1}{m}$ for each $i\in \{1,\ldots,2m\}$, there exists $i,j\in \{1,\ldots,2m\}$ satisfying $\mu(a_i\wedge a_j)\geq \frac{1}{3m^2}.$
\end{lemma}

\begin{proof}
Suppose, towards a contradiction, that we have distinct elements $a_1,\ldots,a_{2m}$ from $B$ such that $\mu(a_i)\geq \frac{1}{m}$ for all $i\in \{1,\ldots,2m\}$ and yet $\mu(a_i\wedge a_j)<\frac{1}{3m^2}$ for all distinct $i,j\in \{1,\ldots,2m\}$.  By the inclusion-exclusion formula, we have 
\begin{alignat}{2}
1\geq \mu(a_1\vee \cdots \vee a_{2m})&\geq \sum_{i=1}^{2m} \mu(a_i)-\sum_{i<j}\mu(a_i\wedge a_j)\notag \\ \notag
 						       &> 2 -\binom{2m}{2}\frac{1}{3m^2}\notag \\
						       &> 2-\frac{2}{3} \notag \\
						       &>1.\notag
\end{alignat}
This contradiction finishes the proof of the lemma.
\end{proof}

\begin{lemma}
Suppose that $B$ is a boolean algebra and $\mu:B\to [0,1]$ is a finitely-additive measure.  Let $k\geq 2$ be a natural number and let $m>0$.  Then there exists a sufficiently large natural number $l:=l(k,m)$ and a positive natural number $c(k,m)$ such that whenever $\{a_1,\ldots,a_l\}$ is a set of $l$ distinct elements of $B$ for which $\mu(a_i)\geq \frac{1}{m}$ for each $i\in \{1,\ldots,l\}$, then there are distinct $i_1,\ldots,i_k\in \{1,\ldots,l\}$ such that $\mu(\bigwedge_{j=1}^k a_{i_j})\geq \frac{1}{c(k,m)}$. 
\end{lemma}

\begin{proof}
By induction on $k$.  The previous lemma shows that the case $k=2$ holds by taking $l(2,m):=2m$ and $c(2,m):=3m^2$.  Now suppose that $k>2$.  We claim that the choices $l(k,m):=2\cdot c(k-1,m)\cdot l(k-1,m)$ and $c(k,m):=c(2,3c(k-1,m)^2)$ are as desired.  Let $l=l(k,m)$ and suppose that $\{a_1,\cdots,a_l\}$ is a set of $l$ distinct elements of $B$.  Then there is a set $\{b_i \ | \ 1\leq i \leq 2\cdot c(k-1,m)\}$ of distinct elements of $B$ such that:
\begin{itemize}
\item each $b_i=\bigwedge_{j=1}^{k-1}a_{i_j}$ for some distinct $i_1,\ldots,i_{k-1}\in \{1,\ldots,l\}$,
\item if $i,i'\in \{1,\ldots,2\cdot c(k-1,m)\}$ are distinct, then $i_j\not= i'_{j'}$ for all $j,j'\in \{1,\ldots,k-1\}$, and
\item $\mu(b_i)\geq \frac{1}{c(k-1,m)}$.  
\end{itemize}
By the case $k=2$, there are $i,j\in \{1,\ldots,2\cdot c(k-1,m)\}$ such that $\mu(b_i\wedge b_j)\geq \frac{1}{3c(k-1,m)^2}$.  This finishes the proof of the lemma.
\end{proof}

\begin{lemma}\label{L:ramsey}
Suppose $B$ is a boolean algebra and $\mu:B\to [0,1]$ is a finitely-additive measure.  Let $k\geq 2$ be a natural number and let $r\in (0,1)$.  Then there exists a sufficiently large natural number $l=l(k,r)$ such that whenever $\{a_1,\ldots,a_l\}$ is a set of $l$ distinct elements of $B$ for which $\mu(a_i)\geq r$ for each $i\in\{1,\ldots,l\}$, then there are distinct $i_1,\ldots,i_k\in \{1,\ldots,l\}$ such that $\mu(\bigwedge_{j=1}^k b_{i_j})>0.$
\end{lemma}

\begin{proof}
Immediate from the preceding lemma.
\end{proof}

\begin{lemma}\label{L:basiclemma2}
Suppose $\varphi(x,y)$ is an $L$-formula, $c$ is a finite tuple from $M$, and $B\subseteq M$ is countable.  Suppose $\varphi(x,c)$ strongly divides over $B$.  Then, for any $r\in (0,1)$, we have $r\dotminus \p\l \varphi(X,\bo{c})\rr$ maximally strongly divides over $\bo{B}$ in the na\"ive sense.  
\end{lemma}

\begin{proof}
Let $k$ be such that $\varphi(x,c)$ strongly $k$-divides over $B$.  Let $l=l(k,r)$ be as in Lemma \ref{L:ramsey}.  We show that $r\dotminus \p \l \varphi(X,\bo{c})\rr$ maximally strongly $l$-divides over $\bo{B}$ in the na\"ive sense.  Let $C_1,\ldots,C_l\models \tp(\bo{c}/\bo{B})$ be $1$-apart.  Then, for almost all $\omega \in \Omega$, $C_1(\omega),\ldots,C_l(\omega)$ are $l$ distinct realizations of $\tp(c/B)$.  Fix $X\in \K^n$, where $n:=|x|$.  Suppose, towards a contradiction, that $r\dotminus \p \l \varphi(X,C_i)\rr=0$ for all $i=1,\ldots,l$.  Then by the defining property of $l$, we see that there are $i_1,\ldots,i_k$ so that $$\{\omega \in \Omega \ | \ \m \models \varphi(X(\omega),C_{i_j}(\omega)), \ j=1,\ldots,k\}$$ has positive measure.  A positive measure subset of these $\omega$'s has the further property that $C_{i_1}(\omega),\ldots,C_{i_k}(\omega)$ are $k$ distinct realizations of $\tp(c/B)$.  This then contradicts the fact that $\varphi(x,c)$ strongly $k$-divides over $B$. 
\end{proof}

\begin{lemma}\label{L:randthdivide}
Suppose $\varphi(x,y)$ is an $L$-formula, $c$ is a tuple from $M$, and $B\subseteq M$ is countable.  Suppose $\varphi(x,c)$ $\th$-divides over $B$.  Then $r\dotminus \p\l \varphi(X,\bo{c})\rr$ maximally $\th$-divides over $\bo{B}$ in the na\"ive sense for any $r\in (0,1]$.  
\end{lemma}

\begin{proof}
Suppose $\varphi(x,c)$ strongly divides over $Bd$.  Then $r\dotminus \p\l \varphi(X,\bo{c})\rr$ maximally strongly divides over $\bo{Bd}$ in the na\"ive sense, whence $r\dotminus \p\l \varphi(X,\bo{c})\rr$ maximally $\th$-divides over $\bo{B}$ in the na\"ive sense.
\end{proof}

\begin{thm}
Suppose $a$ is a tuple from $M$ and $B\subseteq C \subseteq M$ are parameter sets such that $B$ is countable and $C$ is small.  Then $\bo{a} \ind[m\th]_{\bo{B}} \bo{C}$ implies that $a \thind_B C$.
\end{thm}

\begin{proof}
Suppose $\varphi(x,c)\in \tp(a/C)$ $\th$-forks over $B$.  Then there are $L$-formulae $\varphi_1(x,c_1),\ldots,\varphi(x,c_n)$, each of which $\th$-divide over $B$, so that $$M \models \forall x (\varphi(x,c) \to \bigvee_{i=1}^n \varphi_i(x,c_i)).$$  We then have 
$$\z(1-\p \l \varphi(X,\bo{c})\rr) \subseteq \bigcup_{i=1}^n \z(\frac{1}{n}\dotminus \p \l \varphi_i(X,\bo{c_i})\rr),$$ and since each of $\frac{1}{n}\dotminus \p \l \varphi_i(X,\bo{c_i})\rr$ maximally $\th$-divides over $\bo{B}$ in the na\"ive sense by Lemma \ref{L:randthdivide}, we see that $1-\p \l \varphi(X,\bo{c})\rr$ maximally $\th$-forks over $\bo{B}$ in the na\"ive sense.  Since the condition $``1-\p \l \varphi(X,\bo{c})\rr=0"$ is in $\tp(\bo{a}/\bo{C})$, it follows that $\bo{a}\nind[m\th]_{\bo{B}} \bo{C}$ in the na\"ive sense, and hence $\bo{a}\nind[m\th]_{\bo{B}} \bo{C}$ by Lemma \ref{L:na\"ive}.
\end{proof}

\begin{cor}\label{T:realrosy}
Suppose $T^R$ is maximally real rosy.  Then $T$ is real rosy.
\end{cor}

\begin{proof}
Let $a$ be a tuple from $M$ and let $C\subseteq M$ be small.  Since $T^R$ is maximally real rosy, there is a countable $\bo{B}\subseteq \bo{C}$ so that $\bo{a} \ind[m\th]_{\bo{B}} \bo{C}$.  By the preceding theorem, we see that $a\thind_B C$, whence it follows that $T$ is real rosy. 
\end{proof}

We now try to extend Corollary \ref{T:realrosy} to include imaginaries.  We first note that given a $L$-formula $E(x,y)$ which defines an equivalence relation on $M_X$, the $L^R$-formula $\rho_E(X,Y):=\p\l \neg E(x,y)\rr$, defines a pseudometric on $\bo{\m}_X$.  It follows that we can associate to every element $e$ of $M^{\eq}$ an element $\tau(e)$ of $\bo{\m}^{\feq}$.  Indeed, suppose that $c$ is a finite tuple from $M$ and $\pi_E(c)$ is its equivalence class under the $0$-definable equivalence relation $E$.  Let $\pi_{\rho_E}(\bo{c})$ denote the equivalence class of $\bo{c}$ under the equivalence relation $\rho_E=0$.  We then set $\tau(\pi_E(c)):=\pi_{\rho_E}(\bo{c})$.

Suppose $\psi(x_1,\ldots,x_m)$ is an $L^{\eq}$-formula, where each $x_i$ is a variable ranging over $E_i$-equivalence classes.  Fix $r\in [0,1]$.  We then set $\widetilde{\psi}_r(X_1,\ldots,X_m)$ to be the $(L^R)^{\feq}$-formula $$\inf_{X^1}\cdots \inf_{X^m}\max(\max_{1\leq i \leq m} (d(\pi_{\rho_{E_i}}(X^i),X_i)),r\dotminus\p \l\psi^{\eq}(X^1,\ldots,X^m)\rr).$$  (Recall that $\psi^{\eq}(x^1,\ldots,x^m)$ is an $L$-formula such that, for all $a^1,\ldots,a^m$, we have $M^{\eq}\models \psi(\pi_{E_1}(a^1),\ldots,\pi_{E_m}(a^m))$ if and only if $M\models \psi^{\eq}(a^1,\ldots,a^m)$.)

\begin{lemma}\label{L:imagtype2}
Suppose $e\in M^{\eq}$ and $B\subseteq M^{\eq}$ is countable.  Suppose $C\in \k$ is such that $\pi_{\rho_E}(C)\models \tp(\tau(e)/\tau(B))$.  Then $\pi_E(C(\omega))\models \tp(e/B)$ for almost all $\omega \in \Omega$. 
\end{lemma}

\begin{proof}
Fix $\psi(x,b)\in \tp(e/B)$.  Let $e'$ and $b'$ be representatives of the classes of $e$ and $b$ respectively.  Then $M\models \psi^{\eq}(e',b')$, whence $$\p\l \psi^{\eq}(\bo{e'},\bo{b'})\rr=1.$$  It thus follows that $\widetilde{\psi}_1(\tau(e),\tau(b))=0$, so $\widetilde{\psi}_1(\pi_{\rho_E}(C),\tau(b))=0$.  It follows that there are $D,F$ such that $\pi_{\rho_E}(C)=\pi_{\rho_E}(D)$, $\tau(b)=\pi_{\rho}(F)$, and $\p\l \psi^{\eq}(D,F)\rr=1$.  So for almost all $\omega$, $M^{\eq}\models \psi(\pi_E(D(\omega)),\pi(F(\omega))$, whence $M^{\eq}\models \psi(\pi_E(C(\omega)),b)$ for almost all $\omega$.  The lemma follows from the fact that $\tp(e/B)$ is countable.  
\end{proof}

\begin{lemma}
Suppose $c\in M^{\eq}$ and $B\subseteq M^{\eq}$ is a small parameterset such that $c\notin \acl(B)$.  Then $\chi(\tau(c)/\tau(B))=1$.
\end{lemma}

\begin{proof}
Let $(c_i:i<\omega)$ be a nonconstant $B$-indiscernible sequence of realizations of $\tp(c)/B)$.  The lemma follows from the fact that $(\tau(c_i):i<\omega)$ is a $\tau(B)$-indiscernible sequence of realizations of $\tp(\tau(c)/\tau(B))$, which we leave to the reader as an exercise.  
\end{proof}

\begin{lemma}
Suppose $c\in M^{\eq}$ and $B\subseteq M^{\eq}$ is countable.  Further suppose that the $L^{\eq}$-formula $\varphi(x,c)$ strongly divides over $B$.  Then $\widetilde{\varphi}_r(X,\tau(c))$ maximally strongly divides over $\tau(B)$ in the na\"ive sense for any $r\in (0,1)$.
\end{lemma}

\begin{proof}
Let $k$ be such that $\varphi(x,c)$ strongly $k$-divides over $B$.  Let $l=l(k,r)$ be as in Lemma \ref{L:ramsey}.  We show that $\widetilde{\varphi}_r(x,\tau(c))$ maximally strongly $l$-divides over $\tau(B)$ in the na\"ive sense.  Let $\pi_{\rho_E}(C_1),\ldots,\pi_{\rho_E}(C_l)\models \tp(\tau(c)/\tau(B))$ be $1$-apart.  Then for almost all $\omega\in \Omega$, we have that $\pi_E(C_1(\omega)),\ldots,\pi_E(C_l(\omega))$ are $l$ distinct realizations of $\tp(c/B)$.  Suppose, towards a contradiction, that $X\in \k^n$ is such that $\widetilde{\varphi}_r(\pi_\rho(X),\pi_{\rho_E}(C_i))=0$ for all $i=1,\ldots,l$.  Arguing as in Lemma \ref{L:imagtype2}, we see that there are $i_1,\ldots,i_k$ so that $$\{\omega\in \Omega \ | \ \models \varphi(\pi(X(\omega)),\pi_E(C_{i_j}(\omega)), \ j=1,\ldots,k\}$$ has positive measure.  A positive measure subset of these $\omega$'s have the further property that $\pi_E(C_{i_1}(\omega)),\ldots,\pi_E(C_{i_k}(\omega))$ are $k$ distinct realizations of $\tp(c/B)$.  This contradicts the fact that $\varphi(x,c)$ strongly $k$-divides over $B$.  
\end{proof}

\begin{lemma}
Suppose $\varphi(x,y)$ is an $L^{\eq}$-formula, $c$ is a tuple from $M$, and $B\subseteq M$ is countable.  Suppose $\varphi(x,\pi(c))$ $\th$-divides over $\pi(B)$.  Then $\widetilde{\varphi}_r(x,\sigma(\bo{c}))$ maximally $\th$-divides over $\sigma(\bo{B})$ in the na\"ive sense for any $r\in (0,1]$.  
\end{lemma}

\begin{proof}
This follows from the previous lemma in exactly the same way as in the real case.
\end{proof}

\begin{thm}
Suppose $a\in M^{\eq}$ and $B\subseteq C \subseteq M^{\eq}$ are parameter sets such that $B$ is countable and $C$ is small.  Then $\tau(a) \ind[m\th]_{\tau(B)} \tau(C)$ implies that $a \thind_BC$.
\end{thm}

\begin{proof}
Suppose $\varphi(x,c)\in \tp(a/C)$ $\th$-forks over $B$.  Then there are $L^{\eq}$-formulae $\varphi_1(x,c_1),\ldots,\varphi_n(x,c_n)$, each of which $\th$-divide over $B$, so that $M^{\eq} \models \forall x( \varphi(x,\pi(c)) \to \bigvee_{i=1}^n \varphi_i(x,\pi(c_i)))$.  But then 
$$\z(\widetilde{\varphi}_1(X,\tau(c)) \subseteq \bigcup_{i=1}^n \z((\widetilde{\varphi_i})_{\frac{1}{n}}(X,\tau(c_i)))$$ and since each of $(\widetilde{\varphi_i})_{\frac{1}{n}}(X,\tau(c_i))$ maximally $\th$-divides over $\tau(B)$ in the na\"ive sense, we see that $\widetilde{\varphi}_1(X,\tau(c))$ maximally $\th$-forks over $\tau(B)$ in the na\"ive sense.  Since $``\widetilde{\varphi}_1(X,\tau(c))=0"$ is in $\tp(\tau(a)/\tau(C))$, it follows that $\tau(a)\nind[m\th]_{\tau(B)} \tau(C)$ in the na\"ive sense, and hence $\tau(a)\nind[m\th]_{\tau(B)} \tau(C)$ by Lemma \ref{L:na\"ive}.
\end{proof}

\begin{cor}\label{L:rosyimpliesrosy}
Suppose $T^R$ is maximally rosy with respect to finitary imaginaries.  Then $T$ is rosy.
\end{cor}

\begin{proof}
Take $a\in M^{\eq}$ and let $C\subseteq M^{\eq}$ be small.  Since $T^R$ is maximally rosy with respect to finitary imaginaries, there is a countable $\tau(B)\subseteq \tau(C)$ so that $\tau(a) \ind[m\th]_{\tau(B)} \tau(C)$.  By the preceding theorem, we see that $a\thind_BC$, whence it follows that $T$ is rosy. 
\end{proof}

Can we strengthen Corollary \ref{L:rosyimpliesrosy} by weakening the hypothesis from ``$T^R$ is maximally rosy with respect to finitary imaginaries'' to ``$T^R$ is rosy with respect to finitary imaginaries?''  To follow the same style of proof as in this section, it appears that we would need a positive answer to the following Ramsey-theoretic question:

\

Suppose $B$ is a boolean algebra and $\mu:B\to [0,1]$ is a finitely-additive measure.  Let $m_1,m_2\geq 1$ and $k\geq 2$ be fixed.  Does there exist a natural number $l=l(k,m_1,m_2)$ such that whenever $\{a_1,\ldots,a_l\}$ is a set of distinct elements of $B$ and $\{b_{ij} \ | \ 1\leq i<j\leq l\}$ is a set of elements of $B$ such that $\mu(a_i)\geq \frac{1}{m_1}$ for all $i\in \{1,\ldots,l\}$ and $\mu(b_{ij})\geq \frac{1}{m_2}$ for all $i,j\in \{1,\ldots,l\}$ with $i<j$, then there are distinct $i_1,\ldots,i_k\in \{1,\ldots,l\}$ such that $$\mu(\bigwedge_{j=1}^k a_{i_j} \wedge \bigwedge_{i<j\in \{i_1,\ldots,i_k\}} b_{ij})>0?$$

\

However, this question has a negative answer:  If $B\subseteq \mathcal{P}([0,1])$, then each $a_i$ could be a subset of $[0,\frac{1}{2}]$ and each $b_{ij}$ could be a subset of $(\frac{1}{2},1]$.

\end{document}